\documentclass{article}

\usepackage{amsthm}
\usepackage{amsmath, amssymb}
\usepackage{hyperref}
\usepackage{tikz}
\usetikzlibrary{matrix}
\usepackage[matrix, arrow, curve]{xy}

\newtheorem{Thm}{Theorem}[section]
\newtheorem{Def}[Thm]{Definition}
\newtheorem{Lem}[Thm]{Lemma}
\newtheorem{Prop}[Thm]{Proposition}
\newtheorem{Kor}[Thm]{Corollary}
\newtheorem{Rem}[Thm]{Remark}

\title{Cancellation of vector bundles of rank $3$ with trivial Chern classes on smooth affine fourfolds}

\author{Tariq Syed\\
	St. Petersburg Department\\
	Steklov Math. Institute\\
	Fontanka 27, 191023 St. Petersburg (Russia)\\
	tariq.syed@gmx.de}

\date{\today}
\begin{document}

\maketitle

\begin{abstract}
If $n \equiv 0,1~mod~4$, we prove a sum formula $V_{\theta_{0}} (a_{0},a_{R}^{n}) = n \cdot V_{\theta_{0}} (a_{0},a_{R})$ for the generalized Vaserstein symbol whenever $R$ is a smooth affine algebra over a perfect field $k$ with $char(k) \neq 2$ such that $-1 \in {k^{\times}}^{2}$. This enables us to generalize a result of Fasel-Rao-Swan on transformations of unimodular rows via elementary matrices over normal affine algebras of dimension $d \geq 4$ over algebraically closed fields of characteristic $\neq 2$. As a consequence, we prove that any projective module of rank $3$ with trivial Chern classes over a smooth affine algebra of dimension $4$ over an algebraically closed field $k$ with $char(k) \neq 2,3$ is cancellative.
\end{abstract}

\tableofcontents

\section{Introduction}

It is a natural question in the study of projective modules whether a projective module $P$ of rank $d-1$ over a smooth affine algebra $R$ of dimension $d$ over an algebraically closed field is cancellative in general or not. While it is well-known that projective modules of rank $\geq d$ are cancellative (cf. \cite{HB} and \cite{S1}), Aravind Asok and Jean Fasel proved in \cite{AF1} that any projective $R$-module of rank $2$ over a smooth affine algebra $R$ of dimension $3$ over an algebraically closed field $k$ of characteristic $\neq 2$ is cancellative. As for higher dimensions, the special case $P = R^{d-1}$ was settled by Jean Fasel, Ravi Rao and Richard Swan in \cite[Theorem 7.5]{FRS}: Indeed, if $R$ is a normal affine algebra of dimension $d \geq 4$ over an algebraically closed field $k$ with $(d-1)! \in k^{\times}$, then they proved that stably free $R$-modules of rank $d-1$ are free, i.e. $R^{d-1}$ is cancellative. In this paper, we generalize the main ideas in the proof of their result in order to derive some results on elementary transformations of unimodular elements and in order to deduce some cancellation results for projective modules.\\
Before explaining our results, let us briefly sketch the proof given in \cite{FRS}: In order to prove the theorem, one has to show that any unimodular row $a = (a_{1},...,a_{d})$ of length $d$ over $R$ is equivalent to the row $(1,0,...,0)$ with respect to the right action of $GL_{d} (R)$ on the set $Um_{d} (R)$ of unimodular rows of length $d$ over $R$. The following ingredients are used in the proof of \cite[Theorem 7.5]{FRS}:

\begin{itemize}
\item[1)] Suslin's famous theorem on completions of unimodular rows (cf. \cite[Remark after Lemma 2]{S1}): Any unimodular row of length $n$ over a commutative ring of the form $(b_{1}^{(n-1)!},b_{2},...,b_{n})$ can be completed to an invertible matrix.
\item[2)] For $I = \langle a_{4},...,a_{d} \rangle$ and $B=R/I$, one considers the map

\begin{center}
$\mathit{Um}_{3} (B)/E_{3} (B) \rightarrow \mathit{Um}_{d} (R)/E_{d} (R)$, $(\bar{b}_{1}, \bar{b}_{2}, \bar{b}_{3}) \mapsto (b_{1}, b_{2}, b_{3}, a_{4},...,a_{d})$,
\end{center}

which is easily seen to be well-defined.
\item[3)] Swan's Bertini theorem (cf. \cite[Theorem 1.5]{Sw}).
\item[4)] The Vaserstein symbol $V: \mathit{Um}_{3} (B)/E_{3} (B) \rightarrow W_E (B)$, which maps into the so-called elementary symplectic Witt group (cf. \cite[\S 3]{SV}).
\item[5)] The sum formula $n V (\bar{b}_{1}, \bar{b}_{2}, \bar{b}_{3}) = V (\bar{b}^{n}_{1}, \bar{b}_{2}, \bar{b}_{3})$ for any $(\bar{b}_{1}, \bar{b}_{2}, \bar{b}_{3}) \in Um_3 (B)$.
\item[6)] The group $W_E (B)$ is $l$-divisible for any prime $l$ with $gcd(l,char(k))=1$.
\end{itemize}

It follows from 1) that any unimodular row of length $d$ over $R$ of the form $(b_{1}^{(d-1)!}, b_{2}, ..., b_{d})$ is completable to an invertible matrix. In particular, it is sufficient to prove that the given row $a = (a_{1},...,a_{d})$ is equivalent to a row of this form with respect to the right action of $GL_{d} (R)$. It follows from the map in 2) that it suffices to show that $(\bar{a}_{1}, \bar{a}_{2}, \bar{a}_{3})$ equals a row $(\bar{b}_{1}^{(d-1)!}, \bar{b}_{2}, \bar{b}_{3})$ up to the action of $E_{3} (B)$. As a consequence of 3) one can actually assume that $B$ is a smooth threefold over $k$. In case of a smooth affine threefold over an algebraically closed field $k$ with $char(k)\neq 2$, it is known that the Vaserstein symbol in 4) is a bijection (cf. \cite[Corollary 3.5]{RvdK}) and hence induces a group structure on $\mathit{Um}_{3} (B)/E_{3} (B)$. Furthermore, by 5) one has $n V (\bar{b}_{1}, \bar{b}_{2}, \bar{b}_{3}) = V (\bar{b}^{n}_{1}, \bar{b}_{2}, \bar{b}_{3})$ in $W_E (B)$ and hence $n (\bar{b}_{1}, \bar{b}_{2}, \bar{b}_{3}) = (\bar{b}^{n}_{1}, \bar{b}_{2}, \bar{b}_{3})$ for all $(\bar{b}_{1}, \bar{b}_{2}, \bar{b}_{3})$ with respect to the group structure induced by the Vaserstein symbol (cf. \cite[Lemma 7.4]{FRS}).\\
The group $W_E (B)$ is actually a reduced higher Grothendieck-Witt group; using the Gersten-Grothendieck-Witt spectral sequence, one can prove that it is $l$-divisible for primes $l$ with $gcd(l,char(k))=1$ (cf. \cite[Section 6]{FRS}). This implies that for any $j \in \mathbb{N}$ with $gcd (char(k),j) = 1$, the row $(\bar{a}_{1}, \bar{a}_{2}, \bar{a}_{3})$ can be transformed via elementary matrices to a unimodular row of length $3$ over $B$ of the form $(\bar{b}_{1}^{j}, \bar{b}_{2}, \bar{b}_{3})$: In particular, if one takes $j = (d-1)!$, it follows that there is a unimodular row $(\bar{b}_{1}, \bar{b}_{2}, \bar{b}_{3})$ of length $3$ over $B$ such that $(\bar{a}_{1}, \bar{a}_{2}, \bar{a}_{3}) = (d-1)! (\bar{b}_{1}, \bar{b}_{2}, \bar{b}_{3}) = (\bar{b}^{(d-1)}_{1}, \bar{b}_{2}, \bar{b}_{3})$ in $\mathit{Um}_{3} (B)/E_{3} (B)$, which enabled Jean Fasel, Ravi Rao and Richard Swan to conclude their proof in \cite{FRS}.\\
Now let $R$ be a commutative ring. For any projective $R$-module $P$, we denote by $Um (P)$ the set of epimorphisms $P \rightarrow R$. We let $P_{0}$ be a projective $R$-module of rank $2$ with a fixed trivialization $\theta_{0}: R \xrightarrow{\cong} \det (P_{0})$ of its determinant. The main idea of this paper is to extend the steps in the proof of \cite[Theorem 7.5]{FRS} to modules of the form $P_{0} \oplus R^{d-3}$ in such a way that we recover the results in \cite{FRS} whenever $P_{0}$ is free. In this more general situation, we can define maps which are perfectly analogous to the maps in 2) above (cf. Section \ref{2.1}). As for 4), recall also that we defined a generalized Vaserstein symbol

\begin{center}
$V_{\theta_{0}}: Um (P_{0} \oplus R)/E (P_{0} \oplus R) \rightarrow \tilde{V} (R)$
\end{center}

associated to $P_{0}$ and $\theta_{0}$ in \cite{Sy1}; the abelian group $\tilde{V} (R)$ is canonically isomorphic to $W_E (R)$ (cf. \cite[Section 3.B]{Sy1}). Furthermore, it was proven that the generalized Vaserstein symbol is a bijection if $R$ is a regular affine algebra of dimension $3$ over a perfect field $k$ with $c.d. (k) \leq 1$ and $6 \in k^{\times}$ (cf. \cite[Theorem 4.15]{Sy1}). As a matter of fact, it follows from \cite{S1} and \cite[Theorem 2.16]{Sy1} that the generalized Vaserstein symbol is a bijection if $R$ is a regular affine algebra of dimension $3$ over an algebraically closed field $k$.\\
As we have seen above in 5), one of the main ingredients in the proof of \cite[Theorem 7.5]{FRS} was the formula $n V (a_{1},a_{2},a_{3}) = V (a_{1}^{n},a_{2},a_{3})$ for all unimodular rows $(a_{1}, a_{2}, a_{3})$ of length $3$ whenever $R$ is a smooth affine algebra over an algebraically closed field (cf. \cite[Lemma 7.4]{FRS}). It is therefore natural to ask whether an analogous formula holds for the generalized Vaserstein symbol.\\
For this, we let $P_{0}$ be a projective $R$-module of rank $2$ with a fixed trivialization $\theta_{0}$ of its determinant as above. For $n \geq 3$, we denote by $P_{n} = P_{0} \oplus Re_{3} \oplus ... \oplus Re_{n}$ the direct sum of $P_{0}$ and free direct summands of rank $1$ with explicit generators $e_{i}$, $i=3,..,n$. Note that any $R$-linear homomorphism $a: P_{n} \rightarrow R$ can be written as $a = (a_{0}, a_{3},...,a_{n})$, where $a_{0}$ is its restriction to $P_{0}$ and any $a_{i}$, $i=3,...,n$, is the element of $R$ corresponding to the restriction of $a$ to $Re_{i}$, i.e. $a_{i} = a (e_{i})$. We then reinterpret the generalized Vaserstein symbol in the language of motivic homotopy theory, i.e. we show that the generalized Vaserstein symbol is induced by a morphism $\mathcal{V}: E (P_{3})\setminus 0 \rightarrow \mathcal{R}\Omega_{s}^{1} \mathcal{GW}^{3}$ of motivic spaces over $Spec(R)$. By means of this interpretation, we can then prove (cf. Theorem \ref{T3.1}):\\\\
\textbf{Theorem 1.} Let $R$ be a smooth affine algebra over a perfect field $k$ with $char(k) \neq 2$ such that $-1 \in {k^{\times}}^{2}$ and $n \in \mathbb{N}$. If $n \equiv 0,1~mod~4$, then the sum formula $V_{\theta_{0}} (a_{0},a_{R}^{n}) = n \cdot V_{\theta_{0}} (a_{0},a_{R})$ holds for all $(a_{0},a_{R}) \in Um (P_{0} \oplus R)$.\\\\
This theorem enables us to generalize the implicit result of Fasel-Rao-Swan in the proof of \cite[Theorem 7.5]{FRS} that any unimodular row of length $d$ over $R$ can be transformed via elementary matrices to a row of the form $(b_{1}^{j},...,b_{d})$ whenever $R$ is a normal affine algebra of dimension $d \geq 4$ or a smooth affine algebra of dimension $3$ over an algebraically closed field $k$ with $char(k) \neq 2$ and $gcd (char(k),j) = 1$. Using Swan's Bertini theorem, we prove (cf. Theorem \ref{T3.2}):\\\\
\textbf{Theorem 2.} Let $R$ be a normal affine algebra of dimension $d\geq 3$ over an algebraically closed field $k$ with $char (k) \neq 2$; if $d=3$, furthermore assume that $R$ is smooth. Then for any $a \in Um (P_{d})$ and $j \in \mathbb{N}$ with $gcd (char(k),j) = 1$ there is an automorphism $\varphi \in E(P_{d})$ such that $a \varphi$ has the form $b = (b_{0},b_{3}^{j},...,b_{d})$.\\\\
In particular, as soon as an analogue of 1) in the list above holds, one can prove cancellation theorems: If there exists $j \in \mathbb{N}$ with $gcd (char(k),j) = 1$ such that any epimorphism of the form $b = (b_{0},b_{3}^{j},...,b_{d})$ is completable to an automorphism $\psi \in Aut (P_{d})$ (i.e. $b = \pi_{d,d} \psi$), then $P_{d-1} = P_{0} \oplus R^{d-3}$ is cancellative. If $d = 3$ and $j = 2$, then we have explicitly constructed in \cite{Sy1} such an automorphism with determinant $1$ by generalizing a construction given by Krusemeyer in \cite{Kr}. This immediately re-proves the following cancellation theorem (cf. Corollary \ref{C3.3}), which was first proven by Aravind Asok and Jean Fasel in \cite{AF1}:\\\\
\textbf{Theorem 3.} Let $R$ be a smooth affine algebra of dimension $3$ over an algebraically closed field $k$ with $char(k) \neq 2$. Then $Um (P_{0} \oplus R)/SL (P_{0} \oplus R)$ is trivial; in particular, $P_{0}$ is cancellative.\\\\
As another consequence, we can prove the following cancellation theorem (cf. Corollary \ref{C3.4}):\\\\
\textbf{Theorem 4.} Let $R$ be a smooth affine algebra of dimension $4$ over an algebraically closed field $k$ with $char (k) \neq 2,3$ and let $P$ be a projective $R$-module of rank $3$ such that $c_{1} (P) = 0$, $c_{2} (P) = 0$ and $c_{3} (P) = 0$. Then $P$ is cancellative.\\


Let us mention some related work: By using the obstruction theory involving Moore-Postnikov factorizations in $\mathbb{A}^{1}$-homotopy theory Peng Du proved in \cite{D} that finitely generated projective modules (with trivial determinant) of rank $d-1$ over smooth affine algebras of dimension $d\geq 3$ over an algebraically closed field $k$ such that $d! \in k^{\times}$ are cancellative if a deep conjecture on the second non-trivial unstable $\mathbb{A}^{1}$-homotopy sheaf of $\mathbb{A}^{d}\setminus 0$ is true; by using similar methods, Jean Fasel recently proved in \cite{F} that finitely generated projective modules of rank $d-1$ over smooth affine algebras of dimension $d\geq 3$ over an algebraically closed field $k$ such that $d! \in k^{\times}$ are indeed cancellative. Since the methods used in \cite{D} and \cite{F} differ very much from the methods used in this paper, our proof of Theorem 4 as well as all the other results proven in this paper are certainly of independent interest.\\
The organization of this paper is as follows: In Section \ref{2.1}, we study orbits of unimodular elements and we prove some results on transformations of unimodular elements via elementary automorphisms of projective modules. We also generalize the maps in 2) from the list above. Then we give a brief introduction to motivic homotopy theory in Section \ref{2.2} and study in particular the group of endomorphisms of $\mathbb{P}^{1}_{S}$ in the pointed $\mathbb{A}^{1}$-homotopy category $\mathcal{H}_{\bullet} (S)$ over a base scheme $S$ whenever $S$ is the spectrum of a smooth affine algebra over an algebraically closed field $k$ with $char (k) \neq 2$. The next section serves as a brief introduction to higher Grothendieck-Witt groups; in particular, we introduce the elementary symplectic Witt group $W_E (R)$ of any commutative ring $R$ and the group $V (R)$ and study the functoriality of $GW_{1}^{3} (R)$ in terms of the isomorphisms $GW_{1}^{3} (R) \cong V (R) \cong W'_E (R)$. In Section \ref{2.4}, we recall the definition of the generalized Vaserstein symbol given in \cite{Sy1} and give an alternative definition for smooth affine algebras over a perfect field of characteristic $\neq 2$. Finally, we prove the main results of this paper in Section \ref{3}.

\subsection*{Acknowledgements}

The author would like to thank the anonymous referee for suggesting changes which greatly improved the exposition of the paper. Moreover, the author would like to thank both his PhD advisors Jean Fasel and Andreas Rosenschon for many helpful discussions, for their encouragement and for their support. Furthermore, the author would like to thank Christophe Cazanave, Fabien Morel and Anand Sawant for their very helpful comments on motivic homotopy theory. The author was temporarily funded by a grant of the DFG priority program 1786 "Homotopy theory and algebraic geometry". The work is supported by the Ministry of Science and Higher Education of the Russian Federation, agreement № 075–15–2019–1620.

\section{Preliminaries}\label{Preliminaries}

Let $R$ be a commutative ring. A finitely generated projective $R$-module is a direct summand of the free $R$-module $R^n$ for some $n \in \mathbb{N}$. Any finitely generated projective $R$-module has a well-defined rank, which is a function $Spec (R) \rightarrow \mathbb{Z}$; if this function is constant, we just say that $P$ has constant rank (cf. \cite[Section 2.A]{Sy1}). A finitely generated projective $R$-module $P$ is called cancellative if any isomorphism of the form $P \oplus R^n \cong Q \oplus R^n$ for some $n \in \mathbb{N}$ and some $R$-module $Q$ (note that $Q$ is necessarily finitely generated projective as well) implies that $P \cong Q$. For any finitely generated projective $R$-module $P$, we denote by $Aut (P)$ the group of automorphisms of $P$; its subgroup of automorphisms with determinant $1$ is denoted $SL (P)$. Furthermore, for any direct sum $P = \bigoplus_{i=1}^{n} P_{i}$ of finitely generated projective $R$-modules, we let $E (P_{1},...,P_{n})$ (or simply $E (P)$ if the decomposition is understood) be the subgroup of $Aut (P)$ generated by elementary automorphisms, i.e. automorphisms of the form $id_{P} + s$, where $s:P_{j} \rightarrow P_{i}$ is an $R$-linear map for $i \neq j$.\\
We denote by $Um (P)$ the set of epimorphisms $P \rightarrow R$ and by $Um^{t} (P)$ the set of unimodular elements of $P$. Note that $Aut (P)$ and hence any subgroup of $Aut (P)$ act on the right on $Um (P)$ and on the left on $Um^{t} (P)$.\\
There are canonical identifications

\begin{center}
$Um^{t} (P)/Aut (P) \xrightarrow{\cong} Um (P^{\vee})/Aut (P^{\vee})$
\end{center}
and
\begin{center}
$Um^{t} (P)/SL (P) \xrightarrow{\cong} Um (P^{\vee})/SL (P^{\vee})$.
\end{center}

If $P = \bigoplus_{i=0}^{n} P_{i}$ is a direct sum, then obviously $P^{\vee} = \bigoplus_{i=0}^{n} {P}^{\vee}_{i}$ and we have the identification

\begin{center}
$Um^{t} (P)/E (P) \xrightarrow{\cong} Um (P^{\vee})/E (P^{\vee})$.
\end{center}

In this paper, we will study these orbit spaces and will use both interpretations as orbit spaces of the set of epimorphisms or unimodular elements of a projective module.\\
A unimodular row of length $n$ over $R$ is a row vector $(a_{1},...,a_{n})$ with $a_{i} \in R$, $1 \leq i \leq n$, such that $\langle a_{1},...,a_{n}\rangle = R$ as an ideal; we denote by $Um_n (R)$ the set of unimodular rows of length $n$ over $R$. Analogously, a unimodular column of length $n$ over $R$ is a column vector ${(a_{1},...,a_{n})}^{t}$ with $a_{i} \in R$, $1 \leq i \leq n$, such that $\langle a_{1},...,a_{n}\rangle = R$ as an ideal; we denote by $Um_{n}^{t} (R)$ the set of unimodular columns of length $n$ over $R$. If $P = R^{n}$, we naturally identify $Um (P)$ with the set $\mathit{Um}_{n} (R)$ of unimodular rows of length $n$ and $Um^{t} (P)$ with the set $\mathit{Um}_{n}^{t} (R)$ of unimodular columns of length $n$. We also identify $Aut (P)$, $SL (P)$ and $E (P)$ with $GL_{n} (R)$, $SL_{n} (R)$ and $E_{n} (R)$ in this case.\\
Finally, if $P$ is a finitely generated projective $R$-module of constant rank $r$, then the orbit space $Um (P \oplus R)/Aut (P \oplus R)$ corresponds to isomorphism classes of finitely generated projective $R$-modules $Q$ of rank $r$ such that $P \oplus R \cong Q \oplus R$. Similarly, if $P$ moreover has a trivial determinant with fixed trivialization $\theta: \det (P) \xrightarrow{\cong} R$, then the orbit space $Um (P \oplus R)/SL (P \oplus R)$ corresponds to isomorphism classes of finitely generated oriented projective modules which are stably isomorphic to $(P, \theta)$. We refer the reader to \cite[Section 2.D]{Sy1} and \cite[Section 2.A]{Sy2} for more details.

\subsection{Orbits of unimodular elements}\label{2.1}

Let $R$ be a commutative ring. We now fix a finitely generated projective $R$-module $P_0$ of rank $2$. For any $n \geq 3$, let $P_n = P_0 \oplus R e_3 \oplus ... \oplus R e_n$ be the direct sum of $P_0$ and free $R$-modules $R e_i$, $3 \leq i \leq n$, of rank $1$ with explicit generators $e_i$. We denote by $\pi_{k, n}: P_n \rightarrow R$ the projections onto the free direct summands of rank $1$ with index $k = 3, ...,n$. Any element $a \in Um (P_{n})$ can be written as $a = (a_{0},a_{3},...,a_{n})$, where $a_{0}$ is the restriction of $a$ to $P_{0}$ and any $a_{i}$ is the element of $R$ corresponding to the restriction of $a$ to $R e_{i}$ for $3 \leq i \leq n$ respectively.\\
We will now introduce useful maps which allow us to some degree to restrict our study of the orbit spaces $Um (P_{n})/E(P_{n})$ to the orbit spaces of the form $Um (P_{3})/E (P_{3})$. For this, let $n\geq 4$ and $a \in Um (P_{n})$. As explained above, we can write $a$ as $(a_{0},a_{3},...,a_{n})$, where $a_{0}$ is the restriction of $a$ to $P_{0}$ and any $a_{i}$, $i=3,...,n$, is the element of $R$ corresponding to the restriction of $a$ to $Re_{i}$ respectively, i.e. $a_{i} = a (e_{i})$. We denote by $I$ the image of the homomorphism $\tilde{a}=(a_{4},...,a_{n}): \bigoplus_{i=4}^{n} R e_{i} \rightarrow R$; in other words, $I = \langle a_{4},...,a_{d}\rangle$. From now on, we write by abuse of notation $\pi$ for the canonical projection $Q \rightarrow Q/IQ$ for any $R$-module $Q$. We consider the $R/I$-module $P_{3}/IP_{3}$ and naturally identify it with $(P_{0}/IP_{0}) \oplus (R/I)$. Furthermore, we let $Um (P_{3}/IP_{3})$ be the set of $R/I$-linear epimorphisms onto $R/I$. As usual, we may write any $\bar{b} \in Um (P_{3}/IP_{3})$ as $(\bar{b}_{0},\bar{b}_{3})$.

\begin{Prop}\label{ReductionProp}
For any $(a_{0},a_{3},...,a_{n}) \in Um (P_{n})$ (where $n \geq 4$), there exists a well-defined map

\begin{center}
$\Phi (a): Um (P_{3}/I P_{3})/E (P_{3}/I P_{3}) \rightarrow Um (P_{n})/E (P_{n})$
\end{center}

which sends the class of $(\bar{b}_{0},\bar{b}_{3}) \in Um (P_{3}/I P_{3})$ to the class represented by the homomorphism $(b_{0},b_{3},a_{4},...,a_{n}) \in Um (P_{n})$, where $(b_{0},b_{3}): P_{3} \rightarrow R$ is any $R$-linear lift of $(\bar{b}_{0},\bar{b}_{3})$.
\end{Prop}

\begin{proof}
We use the notation from above. For any $\bar{b} \in Um (P_{3}/IP_{3})$, there exists an $R$-linear map $b = (b_{0},b_{3})$ such that the diagram

\begin{center}
$\begin{xy}
  \xymatrix{
          &   P_{3} \ar[d]^{\bar{b} \pi} \ar@{.>}[ld]_{b}  \\
      R \ar@{->>}[r]^{\pi}             &   R/I   
  }
\end{xy}$
\end{center}

commutes, because $P_{3}$ is projective and $R \rightarrow R/I$ is an $R$-linear surjective map. Clearly, the homomorphism $(b_{0},b_{3},a_{4},...,a_{n})$ is then an element of $Um (P_{n})$.\\
Now assume that $b' = (b'_{0},b'_{3})$ is another $R$-linear map such that the diagram above is commutative. Then the $R$-linear map $b-b'$ maps $P_{3}$ into $I$. Thus, as $P_{3}$ is projective, there exists an $R$-linear map $s: P_{3} \rightarrow \bigoplus_{i=4}^{n} R e_{i}$ such that the diagram

\begin{center}
$\begin{xy}
  \xymatrix{
          &   P_{3} \ar[d]^{b-b'} \ar@{.>}[ld]_{s}  \\
      \bigoplus_{i=4}^{n} R e_{i} \ar@{->>}[r]^-{\tilde{a}}             &   I   
  }
\end{xy}$
\end{center}

is commutative. In particular, if we let $\varphi_{s} = id_{P_{n}} + s$ be the elementary automorphism of $P_{n}$ induced by $s$, then $(b'_{0},b'_{3},a_{4},...,a_{n}) \varphi_{s} = (b_{0},b_{3},a_{4},...,a_{n})$.\\
It follows that the assignment $(\bar{b}_{0},\bar{b}_{3}) \mapsto (b_{0},b_{3},a_{4},...,a_{n})$ induces a well-defined map

\begin{center}
$Um (P_{3}/IP_{3}) \rightarrow Um (P_{n})/E (P_{n})$.
\end{center}

Finally, let $\bar{b} \in Um (P_{3}/IP_{3})$ and $\bar{s}: P_{0}/IP_{0} \rightarrow R/I$ and $\bar{t}: R/I \rightarrow P_{0}/IP_{0}$ be $R/I$-linear maps. Then, again since $P_{0}$ and $R$ are projective $R$-modules, there exist $R$-linear lifts $s: P_{0} \rightarrow R$ and $t: R \rightarrow P_{0}$ of $\bar{s}$ and $\bar{t}$ respectively. In particular, $\varphi_{s} = id_{P_{3}} + s$ and $\varphi_{t} = id_{P_{3}} + t$ are lifts of the elementary automorphisms $\varphi_{\bar{s}},\varphi_{\bar{t}}$ of $P_{3}/IP_{3}$ induced by $\bar{s}$ and $\bar{t}$ respectively. If $\bar{b}' = \bar{b} \varphi_{\bar{s}}$ and $\bar{b}'' = \bar{b} \varphi_{\bar{t}}$ and $b: P_{3} \rightarrow R$ is an $R$-linear map which lifts $\bar{b}$, then $b \varphi_{s}$ and $b \varphi_{t}$ are $R$-linear lifts of $\bar{b}'$ and $\bar{b}''$ to $P_{3} \rightarrow R$ respectively. In particular, if we let $b = (b_{0},b_{3})$, $b' = b \varphi_{s} = (b'_{0},b'_{3})$ and $b'' = b \varphi_{t} = (b''_{0}, b''_{3})$, then the classes of $(b_{0},b_{3},a_{4},...,a_{n})$, $(b'_{0},b'_{3},a_{4},...,a_{n})$ and $(b''_{0},b''_{3},a_{4},...,a_{n})$ in $Um (P_{n})/E (P_{n})$ coincide.\\
Altogether, it follows from this that the map above descends to a well-defined map

\begin{center}
$\Phi (a): Um (P_{3}/IP_{3})/E (P_{3}/IP_{3}) \rightarrow Um (P_{n})/E (P_{n})$.
\end{center}
\end{proof}

\begin{Rem}
More generally, let $I_{i} = \langle a_{i+1},...,a_{n}\rangle$ for $3 \leq i \leq n-1$. By repeating the reasoning above, we can prove that there is a well-defined map

\begin{center}
$\Phi_{i} (a): Um (P_{i}/I_{i}P_{i})/E (P_{i}/I_{i}P_{i}) \rightarrow Um (P_{n})/E (P_{n})$
\end{center}

which sends the class of $(\bar{b}_{0},\bar{b}_{3},...,\bar{b}_{i}) \in Um (P_{i}/I_{i}P_{i})$ to the class represented by the homomorphism $(b_{0},...,b_{i},a_{i+1},...,a_{n}) \in Um (P_{n})$, where $(b_{0},b_{3},...,b_{i}): P_{i} \rightarrow R$ is any $R$-linear lift of $(\bar{b}_{0},\bar{b}_{3},...,\bar{b}_{i})$. In particular, $\Phi_{3} (a) = \Phi (a)$.\\
By dualizing the reasoning above, one can also prove that for any unimodular element $a = (a_{0},...,a_{n}) \in P_{n}$, there are analogously defined maps

\begin{center}
$\Phi_{i} (a): Um^{t} (P_{i}/I_{i}P_{i})/E (P_{i}/I_{i}P_{i}) \rightarrow Um^{t} (P_{n})/E (P_{n})$,
\end{center}

where again $I_{i} = \langle a_{i+1},...,a_{n}\rangle$ for $3 \leq i \leq n-1$.
\end{Rem}

The following three lemmas are generalizations to our situations of the corresponding well-known statements when $P_{0} = R^{2}$ (cf. \cite{V}, \cite[Chapter III, Proposition 4.9]{L} and \cite[Lemma 2.7(c)]{SV}):

\begin{Lem}\label{PowerOperations}
Let $n \geq 5$, $a = (a_{0},a_{3},...,a_{n}) \in Um (P_{n})$ and also let $k \in \mathbb{N}$ and $3 \leq i,j \leq n$. Then there exists an automorphism $\varphi \in E (P_{n})$ such that $(a_{0},...,a_{i}^{k},...,a_{n}) \varphi = (a_{0},...,a_{j}^{k},...,a_{n})$.
\end{Lem}

\begin{proof}
Let $J$ denote the image of $a_{0}$. We consider the ring $R/J$ and the unimodular rows $(\bar{a}_{3},...,\bar{a}_{i}^{k},...,\bar{a}_{n})$ and $(\bar{a}_{3},...,\bar{a}_{j}^{k},...,\bar{a}_{n})$. Then it is well-known that there is $\bar{\varphi}' \in E_{n-2} (R/J)$ such that $(\bar{a}_{3},...,\bar{a}_{i}^{k},...,\bar{a}_{n}) \bar{\varphi}' = (\bar{a}_{3},...,\bar{a}_{j}^{k},...,\bar{a}_{n})$ (cf. \cite[Theorem]{V}).\\
We now lift $\bar{\varphi}'$ to an element $\varphi'$ of $E (P_{n})$ which is the identity on $P_{0}$ (lifts of elementary matrices always exist along surjective homomorphisms). If we then set $(b_{0},b_{3},...,b_{n}) = (a_{0},{a}_{3},...,{a}_{i}^{k},...,a_{n}) \varphi'$, there exist $p_{l} \in P_{0}$, $3 \leq l \leq n$, such that $a_{l} - b_{l} = a_{0} (p_{l})$ for $l\neq j$, $3 \leq l \leq n$, and $a_{j}^{k} - b_{j} = a_{0} (p_{j})$. Furthermore, $b_{0} = a_{0}$.\\
Then we let $p: \bigoplus_{l=3}^{n} Re_{l} \rightarrow P_{0}$ be the homomorphism which sends $e_{l}$ to $p_{l}$ for $3 \leq l \leq n$. If we then let $\varphi_{p}$ be the induced element of $E (P_{n})$, the automorphism $\varphi = \varphi' \varphi_{p}$ lies in $E (P_{n})$ and transforms $(a_{0},...,a_{i}^{k},...,a_{n})$ to $(a_{0},...,a_{j}^{k},...,a_{n})$.
\end{proof}

\begin{Lem}\label{PowerOperations2}
Let $n \geq 4$, $a = (a_{0},a_{3},...,a_{n}) \in Um (P_{n})$ and also let $k \in \mathbb{N}$ and $3 \leq i,j \leq n$. Then there exists an automorphism $\varphi \in Aut (P_{n})$ such that $(a_{0},...,a_{i}^{k},...,a_{n}) \varphi = (a_{0},...,a_{j}^{k},...,a_{n})$.
\end{Lem}

\begin{proof}
For ease of notation, let $i=3$ and $j=4$. Then one can check easily that $f (T) = (a_{0}, a_{3}^{k}, a_{4} + T a_{3}, a_{5},...,a_{n})$ gives an element of $Um (P_{n} \otimes_{R} R[T])$.\\
We now prove that $f(T) = f(0) \varphi_{1}$ for some $\varphi_{1} \in Aut (P_{n} \otimes_{R} R[T])$. For this, note that it suffices to prove this over any local ring $(R,\mathfrak{m})$: Such an automorphism $\varphi$ exists if and only if $P = ker(f(T))$ is extended from $R$; by Quillen's patching theorem, it suffices to check the latter statement and consequently also the first statement for each local ring $R_{\mathfrak{m}}$ at any maximal ideal $\mathfrak{m}$ of $R$.\\
In case of any local ring $(R,\mathfrak{m})$, we can clearly assume that $im (a_{0}) \subset \mathfrak{m}$ and $a_{3},a_{5},...,a_{n} \in \mathfrak{m}$ (otherwise the statement would be trivial). But then $a_{4} \in R^{\times}$. In this case, it follows easily that $(a_{3}^{k},a_{4}+Ta_{3}) \in Um_{2} (R[T])$ is a completable unimodular row, because it has length $2$. This clearly implies that $f(T) = f(0) \varphi_{1}$ for some $\varphi_1 \in Aut (P_{n} \otimes_{R} R[T])$, as desired.\\
It follows in particular that $(a_{0}, a_{3}^{k}, a_{4} - a_{3}, ..., a_{n}) = f(-1) = f (0) \varphi_{2}$ for some $\varphi_{2} \in Aut (P_{n})$ (namely $\varphi_{1} (-1)$). Note that $f(0) = (a_{0},a_{3}^{k},a_{4},...,a_{n})$.\\
Now applying the same reasoning of the previous paragraphs to the epimorphism $(a_{0},a_{4},a_{4}-a_{3},a_{5},...,a_{n})$, it follows that there is $\varphi_{3} \in Aut (P_{n})$ such that $(a_{0},a_{4}^{k},-a_{3},...,a_{n}) = (a_{0},a_{4}^{k}, a_{4}-a_{3},...,a_{n}) \varphi_{3}$.\\
Now let $s: Re_{3} \rightarrow Re_{4}, e_{3} \mapsto (a_{4}^{k-1} + a_{4}^{k-2}a_{3} + ... + a_{3}^{k-1}) e_{4}$ and let $\varphi_{s}$ be the induced elementary automorphism of $P_{n}$. Since
\begin{center}
$(a_{4}^{k-1} + a_{4}^{k-2}a_{3} + ... + a_{3}^{k-1})\cdot (a_{4}-a_{3})=a_{4}^{k}-a_{3}^{k}$,
\end{center}
$\varphi_{s}$ transforms $(a_{0}, a_{3}^{k}, a_{4} - a_{3}, ..., a_{n})$ to $(a_{0},a_{4}^{k}, a_{4}-a_{3},...,a_{n})$. Furthermore, let $\psi_{2}$ be the automorphism of $Re_{3} \oplus Re_{4}$ given by the matrix

\begin{center}
$
\begin{pmatrix}
0 & 1 \\
- 1 & 0
\end{pmatrix}
\in E_{2} (R)
$
\end{center}

and let $\psi \in E (P_{n})$ be the automorphism of $P_{n}$ which is $\psi_{2}$ on $R e_{3} \oplus R e_{4}$ and the identity on the other direct summands. Clearly, one has $(a_{0},a_{4}^{k},-a_{3},...,a_{n}) \psi = (a_{0},a_{3},a_{4}^{k},...,a_{n})$. Altogether, it follows from all the previous steps that indeed $(a_{0},a_{3},a_{4}^{k},..., a_{n}) = (a_{0}, a_{3}^{k},a_{4},..., a_{n}) \varphi_{2} \varphi_{s} \varphi_{3} \psi$.
\end{proof}

\begin{Lem}\label{ElementaryCompletion}
Let $a = (a_{0},a_{3},...,a_{n}) \in Um (P_{n})$ with $(a_{4},...,a_{n}) \in Um_{n-3} (R)$. Then there exists $\varphi \in E (P_{n})$ such that $a \varphi = \pi_{n,n}$, i.e. $a \varphi$ is just the projection onto the last direct summand of $P_{n}$.
\end{Lem}

\begin{proof}
As in the previous proof, we let $J$ denote the image of $a_{0}$. We consider the ring $R/J$ and the unimodular row $(\bar{a}_{3},...,\bar{a}_{n})$. Then there is $\bar{\varphi}' \in E_{n-2} (R/J)$ such that $(\bar{a}_{3},...,\bar{a}_{i},...,\bar{a}_{n}) \bar{\varphi}' = (0,...,1)$ (cf. \cite[Lemma 2.7(c)]{SV}).\\
We can then lift $\bar{\varphi}'$ to an element $\varphi'$ of $E (P_{n})$ which is the identity on $P_{0}$. If we set $(b_{0},b_{3},...,b_{n}) = (a_{0},{a}_{3},...,a_{n}) \varphi'$, there exist $p_{l} \in P_{0}$, $3 \leq l \leq n$, such that $a_{l} - b_{l} = a_{0} (p_{l})$ for $3 \leq l \leq n$. Moreover, $b_{0} = a_{0}$.\\
Then we let $p: \bigoplus_{l=3}^{n} Re_{l} \rightarrow P_{0}$ be the homomorphism which sends $e_{l}$ to $p_{l}$ for $3 \leq l \leq n$. If we then let $\varphi_{p}$ be the induced element of $E (P_{n})$, the automorphism $\varphi = \varphi' \varphi_{p}$ lies in $E (P_{n})$ and transforms $(a_{0},...,a_{n})$ to $\pi_{n,n}$.
\end{proof}

\begin{Lem}\label{SwanBertiniArgument}
Assume that $R$ is a normal affine algebra of dimension $d \geq 4$ over an algebraically closed field $k$ with characteristic $\neq 2$; furthermore, let $a = (a_{0},...,a_{d}) \in Um (P_{d})$. Then there exists $\varphi \in E (P_{d})$ such that if we let $a \varphi = (b_{0},...,b_{d})$ and $I = \langle b_{4},...,b_{d} \rangle$, then $R/I$ is either $0$ or a smooth affine algebra of dimension $3$ over $k$.
\end{Lem}

\begin{proof}
Since $R$ is normal, the ideal $J$ of the singular locus of $R$ has height at least $2$ and therefore $\dim (R/J) \leq d-2$. Hence it follows from \cite[Chapter IV, Theorem 3.4]{HB} that $Um (P_{d}/JP_{d}) = \pi_{d,d} E (P_{d}/JP_{d})$ and therefore we can assume that the image of $a_{0}$ and any $a_{i}$ for $3 \leq i \leq d-1$ lie in $J$ and $a_{d} - 1 \in J$.\\
Now let $p = (p_{0},c_{3},...,c_{d}) \in P_{d}$ be a section of $a$, i.e. $a (p) = 1$. Then we consider the unimodular row $\tilde{a}=(a_{0}(p_{0}),a_{3},...,a_{d})$. By Swan's Bertini theorem (cf. \cite[Theorem 1.5]{Sw}), there is an upper triangular matrix $B = (\beta_{i,j})_{2 \leq i,j \leq d}$ (notice the indexing!) of rank $d-1$ over $R$ such that $\tilde{a}B = (a_{0}(p_{0}),a'_{3},...,a'_{d})$ has the property that if $I = \langle a'_{4},...,a'_{d}\rangle$, then $R/I$ is either $0$ or a smooth threefold outside the singular locus of $R$. But by the previous paragraph, it follows that we still have $a'_{d} - 1 \in J$, which means that $R/I$ is either $0$ or a smooth affine threefold over $k$.\\
We now define a homomorphism $s_{0}: \bigoplus_{i=3}^{d} R e_{i} \rightarrow P_{0}$ by $s_{0} (e_{i}) = \beta_{2,i} p_{0}$. Furthermore, we define homomorphisms $s_{l}: \bigoplus_{i=l+1}^{d} R e_{i} \rightarrow R e_{l}$ for each $l$, $3 \leq l \leq d-1$, by $s_{l} (e_{i}) = \beta_{l,i} e_{l}$. Then we let $\varphi_{0}$ and $\varphi_{l}$, $3 \leq l \leq d-1$, be the elementary automorphisms of $P_{d}$ induced by $s_{0}$ and the $s_{l}$ respectively and we define $\varphi = \varphi_{d-1}\circ ...\circ \varphi_{3} \circ \varphi_{0}$. By construction, we have $a \varphi = (a_{0},a'_{3},...,a'_{d})$, which finishes the proof.
\end{proof}

\subsection{Motivic homotopy theory}\label{2.2}

Let $S$ be a regular Noetherian base scheme of finite Krull dimension and let $Sm_{S}$ be the category of smooth separated schemes of finite type over $S$. Furthermore, let $Spc_{S} = \Delta^{op} Shv_{Nis} (Sm_{S})$ (resp. $Spc_{S,\bullet}$) be the category of (pointed) simplicial Nisnevich sheaves over $Sm_{S}$. We write $\mathcal{H}_{s} (S)$ (resp. $\mathcal{H}_{s,\bullet} (S)$) for the (pointed) Nisnevich simplicial homotopy category, which can be obtained as the homotopy category of the injective local model structure on $Spc_{S}$ (resp. $Spc_{S,\bullet}$). Furthermore, we write $\mathcal{H} (S)$ (resp. $\mathcal{H}_{\bullet} (S)$) for the (pointed) $\mathbb{A}^{1}$-homotopy category, which can be obtained as a Bousfield localization of $\mathcal{H}_{s} (S)$ (resp. $\mathcal{H}_{s,\bullet} (S)$); see \cite{MV} for more details. In case of an affine scheme $S = Spec (R)$, we simply write $Spc_{R}$, $Spc_{R,\bullet}$, $\mathcal{H} (R)$ or $\mathcal{H}_{\bullet} (R)$ for the respective categories.\\
Objects of $\mathcal{H} (S)$ (resp. $\mathcal{H}_{\bullet} (S)$) will be referred to as (pointed) spaces. For two spaces $\mathcal{X}$ and $\mathcal{Y}$, we denote by $[\mathcal{X},\mathcal{Y}]_{\mathbb{A}_{S}^{1}} = Hom_{\mathcal{H} (S)} (\mathcal{X},\mathcal{Y})$ the set of morphisms from $\mathcal{X}$ to $\mathcal{Y}$ in $\mathcal{H} (S)$; similarly, for two pointed spaces $(\mathcal{X},x)$ and $(\mathcal{Y},y)$, we denote by $[(\mathcal{X},x), (\mathcal{Y},y)]_{\mathbb{A}_{S}^{1},\bullet} = Hom_{\mathcal{H}_{\bullet} (S)} ((\mathcal{X},x), (\mathcal{Y},y))$ the set of morphisms from $(\mathcal{X},x)$ to $(\mathcal{Y},y)$ in $\mathcal{H}_{\bullet} (S)$. Sometimes we will omit the basepoints from the notation or write $R$ instead of $Spec(R)$ if $S = Spec (R)$ is an affine scheme.\\
If $\mathcal{X}, \mathcal{Y} \in Spc_{S}$, we say that two morphisms $f,g: \mathcal{X} \rightarrow \mathcal{Y}$ are naively $\mathbb{A}_{S}^{1}$-homotopic if there is a morphism $H: \mathcal{X} \times \mathbb{A}^{1}_{S} \rightarrow \mathcal{Y}$ such that $H (-,0) = f$ and $H (-,1) = g$. We denote by $[\mathcal{X}, \mathcal{Y}]_{N}$ the set of equivalence classes of morphisms from $\mathcal{X}$ to $\mathcal{Y}$ under the equivalence relation generated by the relation of naive $\mathbb{A}^{1}_{S}$-homotopies. A (pointed) space $\mathcal{Y}$ is called $\mathbb{A}^{1}_{S}$-fibrant if the unique morphism $\mathcal{Y} \rightarrow \ast = S$ is an $\mathbb{A}^{1}_{S}$-fibration; in fact, for any (pointed) space $\mathcal{Y}$ there is a (pointed) $\mathbb{A}^{1}_{S}$-fibrant space $\mathcal{Y}'$ together with a (pointed) $\mathbb{A}^{1}_{S}$-weak equivalence $\mathcal{Y} \rightarrow \mathcal{Y}'$. If $\mathcal{Y}$ is an $\mathbb{A}^{1}_{S}$-fibrant space and $\mathcal{X}$ is any space, then the relation of naive $\mathbb{A}^{1}_{S}$-homotopies on the set of morphisms from $\mathcal{X}$ to $\mathcal{Y}$ is an equivalence relation and the natural map $[\mathcal{X}, \mathcal{Y}]_{N} \rightarrow [\mathcal{X}, \mathcal{Y}]_{\mathbb{A}^{1}_{S}}$ is a bijection.\\
For any space $\mathcal{X}$, the product functor $\mathcal{X} \times -: Spc_{S} \rightarrow Spc_{S}$ admits a right adjoint $\underline{Hom} (\mathcal{X}, -): Spc_{S} \rightarrow Spc_{S}$; the adjoint pair forms a Quillen pair and therefore induces an adjunction

\begin{center}
$\mathcal{X} \times - : \mathcal{H} (S) \rightleftarrows \mathcal{H} (S): \mathcal{R}\underline{Hom} (\mathcal{X},-)$
\end{center}

on the $\mathbb{A}_{S}^{1}$-homotopy category; here, $\mathcal{R}\underline{Hom} (\mathcal{X}, -)$ denotes a right derived functor of $\underline{Hom} (\mathcal{X},-)$.\\
Like in classical topology, one can define a wedge product $(\mathcal{X},x) \vee (\mathcal{Y},y)$ and a smash product $(\mathcal{X},x) \wedge (\mathcal{Y},y)$ of two pointed spaces $(\mathcal{X},x)$ and $(\mathcal{Y},y)$. For any pointed space $(\mathcal{X},x)$, the functor $(\mathcal{X},x)\wedge-: Spc_{S,\bullet} \rightarrow Spc_{S,\bullet}$ admits a right adjoint $\underline{Hom}_{\bullet} ((\mathcal{X},x),-): Spc_{S,\bullet} \rightarrow Spc_{S,\bullet}$; the adjoint pair forms a Quillen pair and hence descends to an adjunction

\begin{center}
$(\mathcal{X},x)\wedge-: \mathcal{H}_{\bullet} (S) \rightleftarrows \mathcal{H}_{\bullet} (S): \mathcal{R}\underline{Hom}_{\bullet}((\mathcal{X},x),-)$
\end{center}

on the level of the pointed $\mathbb{A}_{S}^{1}$-homotopy category; here, $\mathcal{R}\underline{Hom}_{\bullet}((\mathcal{X},x),-)$ is a right derived functor of $\underline{Hom}_{\bullet}((\mathcal{X},x),-)$. As a particularly interesting special case, one obtains the functor $\Sigma_{s} = S^{1}\wedge-: Spc_{S, \bullet} \rightarrow Spc_{S, \bullet}$, which is called the simplicial suspension functor; its right adjoint $\Omega_{s} = \underline{Hom}_{\bullet} (S^{1},-)$ is called the simplicial loop space functor. The right-derived functor of $\Omega_{s}$ will be denoted $\mathcal{R}\Omega_{s}$. We denote by $\Sigma_{s}^{n}$ and $\Omega_{s}^{n}$ the iterated suspension and loop space functors for any $n \in \mathbb{N}$. For any pointed space $(\mathcal{X},x)$, its simplicial suspension $\Sigma_{s} (\mathcal{X},x) = S^{1} \wedge (\mathcal{X},x)$ has the structure of an $h$-cogroup in $\mathcal{H}_{\bullet} (k)$ (cf. \cite[Definition 2.2.7]{A} or \cite[Section 6.1]{Ho}); in particular, for any pointed space $(\mathcal{Y},y)$, there is a natural group structure on the set $[\Sigma_{s}(\mathcal{X},x), (\mathcal{Y},y)]_{\mathbb{A}^{1}, \bullet}$ induced by the $h$-cogroup structure of $\Sigma_{s} (\mathcal{X},x)$. For any pointed space $(\mathcal{Y},y)$, the space $\mathcal{R}\Omega_{s}(\mathcal{Y},y)$ has the structure of an $h$-group in $\mathcal{H}_{\bullet} (k)$ and hence the set $[(\mathcal{X},x), \mathcal{R}\Omega_{s}(\mathcal{Y},y)]_{\mathbb{A}^{1},\bullet}$ has a natural group structure for any pointed space $(\mathcal{X},x)$ induced by the $h$-group structure of $\mathcal{R}\Omega_{s}(\mathcal{Y},y)$.\\
Furthermore, the functor $Spc_{k} \rightarrow Spc_{k,\bullet}, \mathcal{X} \mapsto X_{+} = X \sqcup \ast$ and the forgetful functor $Spc_{k,\bullet} \rightarrow Spc_{k}$ form a Quillen pair, which will be tacitly used in some proofs of this paper in order to force some spaces to have a basepoint.\\\\
For any base scheme $S$ as above, we let $\mathbb{P}^{1}_{S} = \mathbb{P}^{1} \times_{\mathbb{Z}} S$ and $\mathbb{G}_{m,S} = \mathbb{G}_{m} \times_{\mathbb{Z}} S$, where $\mathbb{P}^{1} = \mathbb{P}^{1}_{\mathbb{Z}} = Proj (\mathbb{Z}[T_{0},T_{1}])$ and $\mathbb{G}_{m} = Spec (\mathbb{Z}[T]_{T})$. If $S = Spec (R)$ is an affine scheme, we simply write $\mathbb{P}^{1}_{R}$ and $\mathbb{G}_{m, R}$ instead of $\mathbb{P}^{1}_{Spec(R)}$ and $\mathbb{G}_{m, Spec (R)}$. The scheme $\mathbb{P}^{1}_{S}$ is canonically pointed by $\infty$, the scheme $\mathbb{G}_{m,S}$ by $1$. It is well-known (cf. \cite[Lemma 3.2.15 and Corollary 3.2.18]{MV}) that there is a canonical pointed $\mathbb{A}_{S}^{1}$-weak equivalence between $\mathbb{P}^{1}_{S}$ and $S^{1} \wedge \mathbb{G}_{m, S}$. Via this identification of $\mathbb{P}^{1}_{S}$ and $S^{1} \wedge \mathbb{G}_{m,S}$ in $\mathcal{H}_{\bullet} (S)$, the space $\mathbb{P}^{1}_{S}$ obtains the structure of an $h$-cogroup. In particular, for any pointed space $(\mathcal{X},x)$, the set $[\mathbb{P}_{S}^{1}, (\mathcal{X},x)]_{\mathbb{A}^{1}_{S},\bullet}$ has a natural group structure.\\
Now let $S=Spec(k)$ be the spectrum of a perfect field $k$ with $char(k) \neq 2$. Then the group $[\mathbb{P}^{1}_{k},\mathbb{P}^{1}_{k}]_{\mathbb{A}^{1}_{k},\bullet}$ has been computed in \cite{C}. In order to explain the computation, we say that two pointed morphisms $f,g:\mathbb{P}^{1}_{k} \rightarrow \mathbb{P}^{1}_{k}$ are naively $\mathbb{A}^{1}_{k}$-homotopic if there is a morphism $H:\mathbb{P}^{1}_{k} \times_{k} \mathbb{A}^{1}_{k} \rightarrow \mathbb{P}^{1}_{k}$ with $H (-,0) = f$, $H (-,1) = g$ and such that $H (\infty,-)=\infty$. We then denote by $[\mathbb{P}^{1}_{k},\mathbb{P}^{1}_{k}]_{N, \bullet}$ the set of equivalence classes of pointed morphisms under the equivalence relation generated by the relation of naive $\mathbb{A}^{1}_{k}$-homotopies.\\
Any pointed morphism $f: \mathbb{P}^{1}_{k} \rightarrow \mathbb{P}^{1}_{k}$ has an associated non-degenerate symmetric bilinear form ${Bez}(f)$ called the B\'ezout form of $f$. We let $MW (k)$ be the Witt monoid of isomorphism classes of non-degenerate symmetric bilinear forms over $k$. The Grothendieck group of $MW (k)$ is the Grothendieck-Witt ring $GW (k)$ of non-degenerate symmetric bilinear forms over $k$. The discriminant induces a well-defined monoid homomorphism $MW (k) \rightarrow k^{\times}/{k^{\times}}^{2}$.\\
For the purpose of this paper, we now summarize the relevant results proven \cite{C}:

\begin{Thm}
The set $[\mathbb{P}^{1}_{k},\mathbb{P}^{1}_{k}]_{N, \bullet}$ can be endowed with a structure of an abelian monoid such that the map $[\mathbb{P}^{1}_{k},\mathbb{P}^{1}_{k}]_{N,\bullet} \rightarrow [\mathbb{P}^{1}_{k},\mathbb{P}^{1}_{k}]_{\mathbb{A}^{1}_{k},\bullet}$ is a group completion. The assignment $(f: \mathbb{P}^{1}_{k} \rightarrow \mathbb{P}^{1}_{k}) \mapsto (Bez(f),\det(Bez(f)))$ induces a monoid isomorphism $[\mathbb{P}^{1}_{k},\mathbb{P}^{1}_{k}]_{N, \bullet} \xrightarrow{\cong} MW (k) \times_{k^{\times}/{k^{\times}}^{2}} k^{\times}$, where the right-hand term is the fiber product with respect to the discriminant map $MW (k) \rightarrow k^{\times}/{k^{\times}}^{2}$ and the projection $k^{\times} \rightarrow k^{\times}/{k^{\times}}^{2}$.
\end{Thm}

\begin{proof}
See \cite[Theorem 3.24]{C} and \cite[Corollary 3.11]{C}.
\end{proof}

As a direct consequence, one then obtains:

\begin{Kor}
We have a group isomorphism

\begin{center}
$[\mathbb{P}^{1}_{k},\mathbb{P}^{1}_{k}]_{\mathbb{A}^{1}_{k},\bullet} \xrightarrow{\cong} GW (k) \times_{k^{\times}/{k^{\times}}^{2}} k^{\times}$.
\end{center}
\end{Kor}

For any $n \in \mathbb{N}$, there is a natural pointed morphism of schemes $\mathbb{G}_{m,k} \rightarrow \mathbb{G}_{m,k}$ induced by the $k$-algebra homomorphism $k[T]_{T} \rightarrow k[T]_{T}, T \mapsto T^{n}$. Taking the smash product with $S^{1}$, we obtain a morphism $\psi^{n}_{k}: S^{1} \wedge \mathbb{G}_{m,k} \rightarrow S^{1} \wedge \mathbb{G}_{m,k}$, which corresponds up to canonical pointed $\mathbb{A}^{1}_{k}$-weak equivalence to a morphism $\mathbb{P}^{1}_{k} \rightarrow \mathbb{P}^{1}_{k}$.

\begin{Kor}
The morphism $\psi^{n}_{k}$ corresponds to $n \cdot id_{S^{1} \wedge \mathbb{G}_{m,k}}$ in the group $[S^{1} \wedge \mathbb{G}_{m,k},S^{1} \wedge \mathbb{G}_{m,k}]_{\mathbb{A}^{1}_{k},\bullet}$.
\end{Kor}

\begin{proof}
The B\'ezout form $Bez(\psi^{n}_{k})$ is given by the $n \times n$-matrix with only $1$'s on the anti-diagonal and $0$'s elsewhere. Its class in $GW (k)$ equals $n_{\epsilon}$, which is given by the formula

\begin{center}
$n_{\epsilon} = \sum_{i=1}^{n} <{(-1)}^{(i-1)}>\;\; \in GW (k)$.
\end{center}

As $\det(Bez(\psi^{n}_{k})) = {(-1)}^{n(n-1)/2}$, it follows that $\psi^{n}_{k}$ corresponds to the pair $(n_{\epsilon},{(-1)}^{n(n-1)/2})$ under the isomorphism $[\mathbb{P}^{1}_{k},\mathbb{P}^{1}_{k}]_{\mathbb{A}^{1}_{k},\bullet} \xrightarrow{\cong} GW (k) \times_{k^{\times}/{k^{\times}}^{2}} k^{\times}$. In particular, if $-1 \in {k^{\times}}^{2}$, i.e. if $-1$ is a square in $k$, and if $n \equiv 0,1~mod~4$, then the morphism $\psi^{n}_{k}$ corresponds to $n \cdot id_{S^{1} \wedge \mathbb{G}_{m,k}}$ in $[S^{1} \wedge \mathbb{G}_{m,k},S^{1} \wedge \mathbb{G}_{m,k}]_{\mathbb{A}^{1}_{k},\bullet}$.
\end{proof}

We now want to prove the latter computation for a more general base scheme. For this, let $k$ be a perfect base field with $char(k) \neq 2$ as in the computation above and let $f: X = Spec (R) \rightarrow Spec(k)$ be a smooth affine scheme of finite type over $k$.\\
If we take $X$ as a base scheme, we again consider the morphism $\mathbb{G}_{m,k} \rightarrow \mathbb{G}_{m,k}$ given by $k[T]_{T} \rightarrow k[T]_{T}, T \mapsto T^{n}$, for all $n \in \mathbb{N}$. Its pullback along the morphism $f: X \rightarrow Spec (k)$ gives a morphism $\mathbb{G}_{m,R} \rightarrow \mathbb{G}_{m,R}$. Taking the smash product with $S^{1}$, we obtain a morphism $\psi^{n}_{R}: S^{1} \wedge \mathbb{G}_{m,R} \rightarrow S^{1} \wedge \mathbb{G}_{m,R}$ in $\mathcal{H}_{\bullet} (R)$.

\begin{Lem}
The morphism $f: X \rightarrow Spec(k)$ induces a group homomorphism $[S^{1} \wedge \mathbb{G}_{m,k},S^{1} \wedge \mathbb{G}_{m,k}]_{\mathbb{A}^{1}_{k},\bullet} \rightarrow [S^{1} \wedge \mathbb{G}_{m,R},S^{1} \wedge \mathbb{G}_{m,R}]_{\mathbb{A}^{1}_{R},\bullet}$.
\end{Lem}

\begin{proof}
There is a restriction functor $f^{\ast}:Spc_{k,\bullet} \rightarrow Spc_{R,\bullet}$ induced by $f$. It follows from \cite[Proposition 3.2.8]{MV} that $f^{\ast}$ commutes with the smash product of pointed spaces, sends $\mathbb{A}^{1}_{k}$-weak equivalences to $\mathbb{A}^{1}_{R}$-weak equivalences and hence descends to a functor $f^{\ast}: \mathcal{H}_{\bullet} (k) \rightarrow \mathcal{H}_{\bullet} (R)$. The functor $f^{\ast}$ sends any smooth $k$-scheme $U$ to its pullback $U \times_{k} R$ along $f$ and similarly sends a morphism $g: U \rightarrow V$ between two $k$-schemes to its pullback $g \times_{k} R: U \times_{k} R \rightarrow V \times_{k} R$; furthermore, it preserves pointed simplicial sets and morphisms between them (i.e. it acts as identity functor on the subcategory of pointed simplicial sets). Hence we obtain a map $[S^{1} \wedge \mathbb{G}_{m,k},S^{1} \wedge \mathbb{G}_{m,k}]_{\mathbb{A}^{1}_{k},\bullet} \rightarrow [S^{1} \wedge \mathbb{G}_{m,R},S^{1} \wedge \mathbb{G}_{m,R}]_{\mathbb{A}^{1}_{R},\bullet}$. As the group structure of both sets is induced by the structure of $S^{1}$ as an $h$-cogroup, the map is clearly a group homomorphism.
\end{proof}

As an immediate consequence of the previous lemma, we obtain:

\begin{Kor}\label{P1Computation}
If $-1 \in {k^{\times}}^{2}$, $n \equiv 0,1~mod~4$ and $\mathcal{X} \in Spc_{R,\bullet}$, then the class of $\psi^{n}_{R}\wedge \mathcal{X}$ in $[S^{1} \wedge \mathbb{G}_{m,R}\wedge\mathcal{X},S^{1} \wedge \mathbb{G}_{m,R}\wedge\mathcal{X}]_{\mathbb{A}^{1}_{R},\bullet}$ equals the class of $n \cdot id_{S^{1} \wedge \mathbb{G}_{m,R} \wedge \mathcal{X}}$.
\end{Kor}

\subsection{Grothendieck-Witt groups}\label{2.3}

In this section we recall some basics about higher Grothendieck-Witt groups, which are a modern version of Hermitian K-theory. The general references of the modern theory are \cite{MS1}, \cite{MS2} and \cite{MS3}.\\
Let $X$ be a scheme with $\frac{1}{2} \in \Gamma (X, \mathcal{O}_{X})$ and let $\mathcal{L}$ be a line bundle on $X$. Then we consider the category $C^{b} (X)$ of bounded complexes of locally free coherent $\mathcal{O}_{X}$-modules. The category $C^{b} (X)$ inherits a natural structure of an exact category from the category of locally free coherent $\mathcal{O}_{X}$-modules by declaring $C'_{\bullet} \rightarrow C_{\bullet} \rightarrow C''_{\bullet}$ to be exact if and only if $C'_{n} \rightarrow C_{n} \rightarrow C''_{n}$ is exact for all $n$. The duality $Hom_{\mathcal{O}_{X}} (-, \mathcal{L})$ induces a duality $\#_{\mathcal{L}}$ on $C^{b} (X)$ and the isomorphism $id \rightarrow Hom_{\mathcal{O}_{X}} (Hom_{\mathcal{O}_{X}} (-, \mathcal{L}), \mathcal{L})$ for locally free coherent $\mathcal{O}_{X}$-modules induces a natural isomorphism of functors $\varpi_{\mathcal{L}}: id \xrightarrow{\sim} \#_{\mathcal{L}}\#_{\mathcal{L}}$ on $C^{b} (X)$. Moreover, the translation functor $T: C^{b} (X) \rightarrow C^{b} (X)$ yields new dualities $\#_{\mathcal{L}}^{j} = T^{j} \#_{\mathcal{L}}$ and natural isomorphisms $\varpi_{\mathcal{L}}^{j} = {(-1)}^{j (j+1)/2} \varpi_{\mathcal{L}}$. We say that a morphism in $C^{b} (X)$ is a weak equivalence if and only if it is a quasi-isomorphism and we denote by $qis$ the class of quasi-isomorphisms. For all $j$, the quadruple $(C^{b} (X), qis, \#_{\mathcal{L}}^{j}, \varpi_{\mathcal{L}}^{j})$ is an exact category with weak equivalences and strong duality (cf. \cite[\S 2.3]{MS2}).\\
Following \cite{MS2}, one can associate a Grothendieck-Witt space $\mathcal{GW}$ to any exact category with weak equivalences and strong duality. The (higher) Grothendieck-Witt groups are then defined to be its homotopy groups:
 
\begin{Def}
For any $i \geq 0$, we let $\mathcal{GW} (C^{b} (X), qis, \#_{\mathcal{L}}^{j}, \varpi_{\mathcal{L}}^{j})$ denote the Grothendieck-Witt space associated to the quadruple $(C^{b} (X), qis, \#_{\mathcal{L}}^{j}, \varpi_{\mathcal{L}}^{j})$ as above. Then we define ${GW}_{i}^{j} (X, \mathcal{L}) = \pi_{i} \mathcal{GW} (C^{b} (X), qis, \#_{\mathcal{L}}^{j}, \varpi_{\mathcal{L}}^{j})$. If $\mathcal{L} = \mathcal{O}_{X}$, we also denote $GW_{i}^{j} (X, \mathcal{O}_{X})$ by $GW_{i}^{j} (X)$. Furthermore, if $X = Spec (R)$, we simply denote $GW_{i}^{j} (X, \mathcal{L})$ or $GW_{i}^{j} (X)$ by $GW_{i}^{j} (R, \mathcal{L})$ or $GW_{i}^{j} (R)$ respectively.
\end{Def}
 
If $S$ is a regular Noetherian affine scheme of finite Krull dimension with $\frac{1}{2} \in \Gamma (S,\mathcal{O}_{S})$, any morphism $f: X \rightarrow Y$ between smooth $S$-schemes induces a pullback morphism $f^{\ast}: {GW}_{i}^{j} (X, \mathcal{L}) \rightarrow {GW}_{i}^{j} (Y, f^{\ast}\mathcal{L})$. The groups $GW_{i}^{j} (X, \mathcal{L})$ are $4$-periodic in $j$. If we let $X = Spec(R)$ be an affine scheme, the groups $GW_{i}^{j} (X)$ coincide with Hermitian K-theory and U-theory as defined by Karoubi (cf. \cite{MK1} and \cite{MK2}), because $\frac{1}{2} \in \Gamma (X, \mathcal{O}_{X})$ by our assumption (cf. \cite[Remark 4.13]{MS1} and \cite[Theorems 6.1-2]{MS3}).\\
In particular, we have isomorphisms $K_{i}O (R) = GW_{i}^{0} (R)$, $_{-1}U_{i} (R) = GW_{i}^{1} (R)$, $K_{i}Sp (R) = GW_{i}^{2} (R)$ and $U_{i} (R) = GW_{i}^{3} (R)$.\\
In this paper, the group of particular interest is $GW_{1}^{3} (X) = U_{1} (R)$. We are now going to discuss various presentations of this group.
First of all, we refer the reader to \cite[Section 2.A]{Sy1} for a short introduction to (non-degenerate) skew-symmetric forms, skew-symmetric isomorphisms, (non-degenerate) alternating forms and alternating isomorphisms. As usual, we identify non-degenerate alternating (skew-symmetric) forms on $R^{2n}$ with invertible alternating (skew-symmetric) $2n \times 2n$-matrices (\cite[\S 3]{SV}).\\
For any commutative ring $R$, the group $V_{\mathit{free}} (R)$ is the quotient of the free abelian group on isometry classes of triples $(P, g, f)$ (with $P$ a finitely generated free $R$-module and $f,g$ alternating isomorphisms on $P$) modulo the subgroup generated by the relations
 
\begin{itemize}
\item $[P \oplus P', g \perp g', f \perp f'] = [P, g, f] + [P', g', f']$ for alternating isomorphisms $f,g$ on $P$ and $f',g'$ on $P'$;
\item $[P, f_0, f_1] + [P, f_1, f_2] = [P, f_0, f_2]$ for alternating isomorphisms $f_{0}$, $f_{1}$ and $f_{2}$ on $P$.
\end{itemize}

Similarly, for any commutative ring $R$, the group $V (R)$ is the quotient of the free abelian group on isometry classes of triples $(P, g, f)$ (with $P$ a finitely generated projective $R$-module and $f,g$ alternating isomorphisms on $P$) modulo the subgroup generated by the relations
 
\begin{itemize}
\item $[P \oplus P', g \perp g', f \perp f'] = [P, g, f] + [P', g', f']$ for alternating isomorphisms $f,g$ on $P$ and $f',g'$ on $P'$;
\item $[P, f_0, f_1] + [P, f_1, f_2] = [P, f_0, f_2]$ for alternating isomorphisms $f_{0}$, $f_{1}$ and $f_{2}$ on $P$.
\end{itemize}

In this context, recall from \cite[Section 2.A]{Sy1} that it is equivalent to speak of skew-symmetric or alternating isomorphisms if $2 \in R^{\times}$. Both groups $V_{\mathit{free}}(R)$ and $V (R)$ are functorial in the sense that any ring homomorphism $\varphi: R \rightarrow S$ induces group homomorphisms $\varphi^{\ast}:V_{\mathit{free}}(R) \rightarrow V_{\mathit{free}}(S)$ and $\varphi^{\ast}:V(R) \rightarrow V(S)$ given by the formula $[P,g,f] \mapsto [P \otimes_{R} S, g \otimes_{R} S, f \otimes_{R} S]$.\\
Finally, for any $n \in \mathbb{N}$, we let $A_{2n} (R)$ be the set of alternating invertible $2n\times2n$-matrices. We inductively define an element $\psi_{2n} \in  A_{2n} (R)$ by setting
 
\begin{center}
$\psi_2 =
\begin{pmatrix}
0 & 1 \\
- 1 & 0
\end{pmatrix}
$
\end{center}
 
\noindent and $\psi_{2n+2} = \psi_{2n} \perp \psi_2$. For any $m < n$, there is an embedding of $A_{2m} (R)$ into $A_{2n} (R)$ given by $M \mapsto M \perp \psi_{2n-2m}$. We denote by $A (R)$ the direct limit of the sets $A_{2n} (R)$ under these embeddings. We then say that two alternating invertible matrices $M \in A_{2m} (R)$ and $N \in A_{2n} (R)$ are equivalent, $M \sim N$, if there is an integer $s \in \mathbb{N}$ and a matrix $E \in E_{2n+2m+2s} (R)$ such that

\begin{center}
$M \perp \psi_{2n+2s} = E^{t} (N \perp \psi_{2m+2s}) E$.
\end{center}

This defines an equivalence relation on $A (R)$ and the set of equivalence classes is denoted $W'_E (R) = A (R)/{\sim}$. It was proven in \cite[\S 3]{SV} that the orthogonal sum of alternating invertible matrices endows $W'_E (R)$ with the structure of an abelian group. The subgroup $W_{E} (R) \subset W'_{E} (R)$ corresponding to alternating invertible matrices with Pfaffian $1$ is called the elementary symplectic Witt group. The group $W'_{E}(R)$ is functorial in the sense that any ring homomorphism $\varphi: R \rightarrow S$ induces a group homomorphisms $\varphi^{\ast}: W'_{E} (R) \rightarrow W'_{E} (S)$ by simply applying $\varphi$ to the entries of invertible alternating matrices.\\
Let us now explain how the groups defined above can be identified. As explained in \cite[Section 3.B]{Sy1} the obvious map $V_{\mathit{free}} (R) \rightarrow V (R), [P,g,f] \mapsto [P,g,f]$ is an isomorphism for any commutative ring. It is also argued in \cite[Section 3.B]{Sy1} that the group $V_{\mathit{free}} (R)$ is canonically isomorphic to the group $W'_{E} (R)$ (this was already observed in \cite{FRS} with the additional assumption that $2 \in R^{\times}$); the isomorphism is given by the formula $M \mapsto [R^{2n},\psi_{2n},M]$ for any alternating invertible $2n\times2n$-matrix $M$ (which is also interpreted as an alternating isomorphism on $R^{2n}$). It follows that the same assignment also gives a canonical isomorphism $W'_E (R) \cong V(R)$ for any commutative ring. Furthermore, it follows directly from the descriptions of the functoriality above that the isomorphisms $W'_E (R) \cong V(R)$ are natural with respect to ring homomorphisms. We denote by $\tilde{V}(R)$ the subgroup of $V (R)$ corresponding to $W_E (R)$ under this isomorphism. Finally, it is argued in \cite{FRS} that there is a natural isomorphism between $GW_{1}^{3} (R)$ and the group $V_{\mathit{free}} (R)$.
\\\\
Now let $S$ be a regular Noetherian affine scheme of finite Krull dimension with $\frac{1}{2} \in \Gamma (S,\mathcal{O}_{S})$. Then it is known that higher Grothendieck-Witt groups of smooth separated schemes of finite type over $S$ are representable in the pointed $\mathbb{A}_{S}^{1}$-homotopy category $\mathcal{H}_{\bullet} (S)$ as defined by Morel and Voevodsky (cf. \cite{JH}). More precisely, if we let $X$ be a smooth separated scheme of finite type over $S$, it is shown that there are pointed spaces $\mathcal{GW}^{j}$ and natural isomorphisms
 
\begin{center}
$[\Sigma_{s}^{i} X_{+}, \mathcal{GW}^{j}]_{\mathbb{A}_{S}^{1}, \bullet} \cong GW_{i}^{j} (X)$.
\end{center}

In particular, we have identifications

\begin{center}
$[X, \mathcal{R}\Omega_{s}^{i}\mathcal{GW}^{j}]_{\mathbb{A}_{S}^{1}} \cong GW_{i}^{j} (X)$.
\end{center}

In particular, it follows immediately from this that the pullback morphism $f^{\ast}:GW_{i}^{j} (X) \rightarrow GW_{i}^{j} (Y)$ induced by any morphism $f: Y \rightarrow X$ of smooth separated schemes of finite type over $S$ can be studied by means of this representability result.\\
Let us consider the case $i=1$,$j=3$. It was proven in \cite{ST} that there is an $\mathbb{A}^{1}_{S}$-weak equivalence $\mathcal{R}\Omega_{s}^{1}\mathcal{GW}^{3} \simeq_{\mathbb{A}^{1}_{S}} GL/Sp$, where $GL$ and $Sp$ denote the infinite linear and symplectic groups (over $R$) respectively. Now let $S = Spec (R)$, where $R$ is a smooth affine algebra over any field $k$ with $char(k)\neq 2$. We let $A_{2n}$ denote the scheme (over $R$) of skew-symmetric invertible $2n\times2n$-matrices. For any $n \in \mathbb{N}$, one can define a morphism $GL_{2n}/Sp_{2n} \rightarrow A_{2n}$ by $M \mapsto M^{t}\psi_{2n}M$. By the same reasoning as in \cite[Section 2.3.2]{AF3}, these morphisms are isomorphisms and hence induce an isomorphism between $GL/Sp$ and $A = colim_{n} A_{2n}$ (transition maps are defined by adding $\psi_{2}$). If $B$ is a smooth $R$-algebra and $Y = Spec (B)$, the obvious map $A(B) \rightarrow [Y, A]_{\mathbb{A}^{1}_{R}}$ induces the identification $W'_E (B) \cong GW_{1}^{3} (B)$ (cf. \cite[Theorem 8.4]{ST} and \cite[Section 2.3.2]{AF3}). In particular, the functoriality of $GW_{1}^{3} (B)$ is just given by the functoriality of $W'_E (B)$ and, by our reasoning above, can also be understood in terms of $V(B)$:\\
More explicitly, assume that $Y = Spec (B)$ and $Z = Spec (C)$ are all smooth affine schemes over $S = Spec (R)$. In the special case $i=1$,$j=3$, we can make the functoriality more explicit in terms of the groups $V (B)$ and $V (C)$. Any morphism $Z \rightarrow Y$ over $Spec (R)$ then corresponds to an $R$-algebra homomorphism $f:B \rightarrow C$. If $P$ is a finitely generated projective $B$-module with alternating isomorphisms $f$ and $g$, then the class of the triple $[P,g,f] \in V (B)$ is sent under the pullback morphism $f^{\ast}$ to $[P\otimes_{B}C,g\otimes_{B}C,f\otimes_{B}C] \in V (C)$.

\subsection{The generalized Vaserstein symbol}\label{2.4}

Let $R$ be a commutative ring and $P_0$ be a projective $R$-module of rank $2$. We assume that $P_0$ admits a trivialization $\theta_{0}: R \rightarrow \det(P_0)$ of its determinant. Let us recall the definition of the generalized Vaserstein symbol associated to the fixed trivialization $\theta_{0}$ from \cite[Section 4.B]{Sy1}:

\begin{Def}
We denote by $\chi_0$ the canonical non-degenerate alternating form on $P_0$ given by $P_0 \times P_0 \rightarrow R, (p,q) \mapsto \theta_{0}^{-1} (p \wedge q)$.
\end{Def}

Now let $Um (P_0 \oplus R)$ be the set of epimorphism $P_0 \oplus R \rightarrow R$. Any element $a$ of $Um (P_0 \oplus R)$ gives rise to an exact sequence of the form

\begin{center}
$0 \rightarrow P(a) \rightarrow P_0 \oplus R \xrightarrow{a} R \rightarrow 0$,
\end{center}

\noindent where $P(a) = ker (a)$. Any section $s: R \rightarrow P_0 \oplus R$ of $a$ determines a canonical retraction $r_{s}: P_0 \oplus R \rightarrow P(a)$ given by $r_{s}(p)= p - s a(p)$ and an isomorphism $i_{s}: P_0 \oplus R \rightarrow P(a) \oplus R$ given by $i_{s}(p)  = a(p) + r_{s}(p)$.\\
The exact sequence above yields an isomorphism $\det(P_0) \cong \det (P(a))$ (independent of $s$) and therefore an isomorphism $\theta : R \rightarrow \det (P(a))$ obtained by composing with $\theta_0$.

\begin{Def}
We denote by $\chi_a$ the non-degenerate alternating form on $P(a)$ given by $P(a) \times P(a) \rightarrow R, (p,q) \mapsto \theta^{-1} (p \wedge q)$.
\end{Def}

Altogether, we obtain a non-degenerate alternating form

\begin{center}
$V (a,s) = {(i_{s} \oplus 1)}^{t} (\chi_a \perp \psi_2) {(i_{s} \oplus 1)}$
\end{center}

on $P_{0} \oplus R^{2}$.

\begin{Rem}\label{projection}
Note that if we let $a:P_{0} \oplus R \rightarrow R$ be the projection onto the last coordinate and $s:R \rightarrow P_{0} \oplus R$ be the inclusion of the last direct summand, then $V(a,s)$ is just $\chi_{0} \perp \psi_{2}$.
\end{Rem}

In general, the non-degenerate alternating form $V(a,s)$ depends on the section $s$ of $a$. Nevertheless, this non-degenerate alternating form gives rise to the definition of the generalized Vaserstein symbol:

\begin{Prop}\label{Vaserstein-Prop}
Assigning to $a \in Um (P_{0} \oplus R)$ the element

\begin{center}
$V_{\theta_{0}} (a) = [P_0 \oplus R^2, \chi_0 \perp \psi_2, {(i_{s} \oplus 1)}^{t} (\chi_a \perp \psi_2) {(i_{s} \oplus 1)}]$
\end{center}

induces a well-defined map $V_{\theta_{0}}: Um (P_{0} \oplus R)/E (P_{0} \oplus R) \rightarrow \tilde{V} (R)$.
\end{Prop}

\begin{proof}
See \cite[Theorem 4.1]{Sy1} for a proof that the element $V_{\theta_{0}} (a)$ does not depend on the choice of the section $s$, see \cite[Lemma 4.2]{Sy1} for a proof that the element $V_{\theta_{0}} (a)$ lies in $\tilde{V}(R)$ and see \cite[Theorem 4.3]{Sy1} for a proof that the induced well-defined map $V_{\theta_{0}}: Um (P_{0} \oplus R) \oplus R) \rightarrow \tilde{V} (R)$ factors through the action of $E (P_{0} \oplus R)$.
\end{proof}

\begin{Def}
The well-defined map $V_{\theta_{0}}: Um (P_{0} \oplus R)/E (P_{0} \oplus R) \rightarrow \tilde{V} (R)$ from Proposition \ref{Vaserstein-Prop} is called the generalized Vaserstein symbol associated to $P_0$ and $\theta_{0}$. If there is no ambiguity, we denote $V_{\theta_{0}}$ simply by $V$.
\end{Def}

It follows from \cite[Proposition 2.16]{Sy1}, \cite{S1} and \cite{B} that this map is a bijection whenever $R$ is a regular affine algebra over any algebraically closed field $k$ or over any perfect field $k$ with $c.d.(k) \leq 1$ and $6 \in k^{\times}$.\\
Now let $R$ be any commutative ring and $P_{0}$ be as usual a projective $R$-module of rank $2$ with a fixed trivialization $\theta_{0}$ of its determinant. Furthermore, we let $f: R \rightarrow B$ and $g: B \rightarrow C$ be ring homomorphisms. Then we have canonical maps

\begin{center}
$f_{Um}^{\ast}:Um (P_{0} \oplus R) \rightarrow Um ((P_{0} \otimes_{R} B) \oplus B)$
\end{center}

and

\begin{center}
$g_{Um}^{\ast}:Um ((P_{0} \otimes_{R} B) \oplus B) \rightarrow Um ((P_{0} \otimes_{R} C) \oplus C)$.
\end{center}

Furthermore, the $B$-module $P_{0} \otimes_{R} B$ and the $C$-module $P_{0} \otimes_{R} C$ have trivial determinants; their trivializations are given by $\theta_{0} \otimes_{R} B$ and $\theta_{0} \otimes_{R} C$ respectively. Finally, recall that we have a group homomorphism $g^{\ast}:V (B) \rightarrow V (C)$ which sends any class $[P,\chi_{1},\chi_{2}]$ in $V (B)$ to the class $[P \otimes_{B} C, \chi_{1} \otimes_{B} C, \chi_{2} \otimes_{B} C]$ in $V (C)$.

\begin{Lem}\label{Naturality}
We have an equality $V_{\theta_{0} \otimes_{R} C} (g_{Um}^{\ast} (a)) = g^{\ast} (V_{\theta \otimes_{R} B} (a))$ for any $a \in Um ((P_{0} \otimes_{R} B)\oplus B)$.
\end{Lem}

\begin{proof}
This is just a functoriality statement for $R$-algebra homomorphisms. If $s: B \rightarrow (P_{0} \otimes_{R} B)\oplus B$ is a section of $a$, then $s$ clearly induces a section $s \otimes_{B} C$ of $a \otimes_{B} C \in Um ((P_{0} \otimes_{R} C) \oplus C)$. Then we let $P (a) = ker (a)$, $P (a \otimes_{B} C) = ker (a \otimes_{B} C)$ and moreover we let $i_{s}: (P_{0} \otimes_{R} B)\oplus B \rightarrow P (a) \oplus B$ and $i_{s \otimes_{B} C}: (P_{0} \otimes_{R} C)\oplus C \rightarrow P (a \otimes_{B} C) \oplus C$ be the isomorphisms induced by $s$ and $s \otimes_{B} C$. As above, we denote by $V (a,s)$ and $V (a\otimes_{B} C, s \otimes_{B} C)$ the corresponding non-degenerate alternating forms ${(i_s \oplus 1)}^{t} (\chi_a \perp \psi_2) {(i_s \oplus 1)}$ and ${(i_{s \otimes_{B} C} \oplus 1)}^{t} (\chi_{a \otimes_{B} C} \perp \psi_2) {(i_{s \otimes_{B} C} \oplus 1)}$ from the definition of the generalized Vaserstein symbol. As usual, we let $P_{4} = P_{0} \oplus R^{2}$. Then, under the isomorphism $(P_{4} \otimes_{R} B)\otimes_{B} C \xrightarrow{\cong} P_{4}\otimes_{R} C$, it is routine to check that the alternating form $V (a,s)\otimes_{B} C$ corresponds to $V (a\otimes_{B} C, s\otimes_{B} C)$. In particular, by Remark \ref{projection}, $((\chi_{0}\otimes_{R} B)\otimes_{B} C) \perp \psi_{2}$ corresponds to $(\chi_{0}\otimes_{R} C) \perp \psi_{2}$. This proves the lemma.
\end{proof}

We now fix a smooth affine algebra $R$ over a perfect field $k$ with $char(k) \neq 2$ as a base ring and give an alternative description of the generalized Vaserstein symbol for smooth affine algebras over the base ring $R$.\\
For this, we start with a few general remarks: We fix finitely generated projective $R$-modules $P$ and $Q$ such that $P \oplus Q = R^{n}$ for some $n \in \mathbb{N}$. Furthermore, we denote by $Sym (P)$, $Sym (Q)$ and $Sym (R^{n}) = R[X_{1},...,X_{n}]$ the symmetric $R$-algebras of $P$, $Q$ and $R^{n}$ respectively. Next we set $E (P) = Spec (Sym (P))$, $E(Q) = Spec (Sym (Q))$ and identify $\mathbb{A}_{R}^{n}$ with $Spec (Sym(R^{n}))$. Note that the inclusions $i_{P}$, $i_{Q}$ of $P$ and $Q$ into $R^{n}$ and the projections $\pi_{P}$, $\pi_{Q}$ of $R^{n}$ onto $P$ and $Q$ respectively induce $R$-algebra homomorphisms between the corresponding symmetric algebras.\\
We denote by $\langle P \rangle$ and $\langle Q \rangle$ the ideals in $Sym (P)$ and $Sym (Q)$ generated by the homogeneous elements of degree $\geq 1$ and denote by $0$ their corresponding closed subschemes of $E (P)$ and $E (Q)$. Geometrically, $E (P) \setminus 0$ is the complement of the zero section of the vector bundle over $Spec(R)$ associated to $P$. By abuse of notation, we also denote by $\langle P \rangle$ and $\langle Q \rangle$ the ideals generated by their images in $Sym (R^{n})$. Note that $Sym (R^{n})/\langle Q \rangle \cong Sym (P)$.\\
Now let $S^{R}_{2n-1} = R [X_{1},...,X_{n},Y_{1},...,Y_{n}]/\langle \sum_{i=1}^{n} X_{i}Y_{i} - 1 \rangle$ and furthermore let $Q^{R}_{2n-1} = Spec (S^{R}_{2n-1})$. Then the $R$-algebra homomorphism

\begin{center}
$i_{n}:Sym (R^{n}) = R[X_{1},...,X_{n}] \rightarrow S^{R}_{2n-1}, X_{i} \mapsto \bar{X}_{i}$ 
\end{center}

induces a Zariski-locally trivial morphism of schemes 

\begin{center}
$pr_{n}:Q^{R}_{2n-1} \rightarrow \mathbb{A}_{R}^{n}\setminus 0$ 
\end{center}

with fibers isomorphic to $\mathbb{A}_{R}^{n-1}$.\\
Again by abuse of notation, we will denote by $\langle Q \rangle$ the ideal generated by the image of $\langle Q \rangle \subset Sym (Q)$ under the map $Sym (Q) \rightarrow Sym (R^{n}) \xrightarrow{i_{n}} S^{R}_{2n-1}$; furthermore, we define $\bar{S}^{R}_{2n-1}= S^{R}_{2n-1}/\langle Q \rangle$ and $\bar{Q}^{R}_{2n-1} = Spec(\bar{S}^{R}_{2n-1})$. One can check easily that the composite of $R$-algebra homomorphisms

\begin{center}
$\bar{i}_{n}:Sym (P) \rightarrow Sym (R^{n}) \xrightarrow{i_{n}} S^{R}_{2n-1} \rightarrow \bar{S}^{R}_{2n-1}$ 
\end{center}

induces a Zariski-locally trivial morphism of schemes

\begin{center}
$\bar{pr}_{n}:\bar{Q}^{R}_{2n-1} \rightarrow E (P)\setminus 0$
\end{center}

with fibers isomorphic to $\mathbb{A}_{R}^{n-1}$. It follows that $\bar{Q}^{R}_{2n-1}$ is a smooth scheme over $Spec (R)$ and $\bar{pr}_{n}$ is an $\mathbb{A}_{R}^{1}$-weak equivalence. Locally over a point of $Spec (R)$ the map $\bar{pr}_{n}$ just corresponds to $pr_{r}$, where $r$ is the rank of $P$.\\
Now let $B$ be a smooth affine algebra over $R$. Then one can check easily that there are natural bijections

\begin{center}
$Hom_{R-\mathfrak{Alg}} (Sym (P), B) \xrightarrow{\cong} Hom_{B-\mathfrak{Mod}} (P\otimes_{R} B, B)$
\end{center}

and

\begin{center}
$Hom_{R-\mathfrak{Alg}} (\bar{S}^{R}_{2n-1},B) \xrightarrow{\cong} \{(a,s) \in Um (P\otimes_{R} B)\times B^{n}|a(\pi_{P \otimes_{R} B}(s))=1\}$.
\end{center}

We apply the previous paragraphs now to the case $P = P_{3} = P_{0} \oplus R$ (with $P_{0}$ a projective $R$-module of rank $2$ with a fixed trivialization $\theta_{0}$ of its determinant as usual). The epimorphism $\pi_{R}: P_{0} \oplus R \rightarrow R$ with section $(0,1) \in P_{0} \oplus R$ induces basepoints $Spec (R) \rightarrow \bar{Q}_{2n-1}$ and $Spec (R) \rightarrow E(P)\setminus 0$. The morphism $\bar{pr}_{n}: \bar{Q}^{R}_{2n-1} \rightarrow E (P)\setminus 0$ is then a pointed morphism; in particular, it has an inverse $\bar{pr}_{n}^{-1}$ in $\mathcal{H}_{\bullet} (R)$. Forgetting the basepoints, we may also interpret this morphism as a morphism in $\mathcal{H} (R)$.\\
The identity of $\bar{S}^{R}_{2n-1}$ corresponds to an epimorphism $a: P_{3} \otimes_{R} \bar{S}^{R}_{2n-1} \rightarrow \bar{S}^{R}_{2n-1}$ with a section $s \in P_{3} \otimes_{R} \bar{S}^{R}_{2n-1}$ and an element $t \in Q \otimes_{R} \bar{S}^{R}_{2n-1}$. Therefore the identity on $\bar{Q}^{R}_{2n-1}$ determines a well-defined generalized Vaserstein symbol (with respect to the fixed trivialization $\theta_{0} \otimes_{R} \bar{S}^{R}_{2n-1}$ of $\det (P_{0}) \otimes_{R} \bar{S}^{R}_{2n-1}$)

\begin{center}
$V (a) \in V (\bar{S}^{R}_{2n-1}) \cong [\bar{Q}^{R}_{2n-1},\mathcal{R}\Omega_{s}^{1}\mathcal{GW}^{3}]_{\mathbb{A}_{R}^{1}}$,
\end{center}

which by Section \ref{2.3} corresponds to a morphism $\bar{Q}^{R}_{2n-1} \rightarrow \mathcal{R}\Omega_{s}^{1}\mathcal{GW}^{3}$ in $\mathcal{H} (R)$; we will denote this morphism by $\mathcal{V}$. Furthermore, the composite

\begin{center}
$E(P)\setminus 0 \xrightarrow{\bar{pr}_{n}^{-1}} \bar{Q}^{R}_{2n-1} \xrightarrow{\mathcal{V}} \mathcal{R}\Omega_{s}^{1}\mathcal{GW}^{3}$
\end{center}

defines a morphism $E (P)\setminus 0 \rightarrow \mathcal{R}\Omega_{s}^{1}\mathcal{GW}^{3}$ in $\mathcal{H} (R)$, which we again denote by $\mathcal{V}$. This morphism is a universal generalized Vaserstein symbol for $P_{0}$ and $\theta_{0}$. 

\begin{Lem}
Assume that $B$ is a finitely generated smooth affine algebra over $R$, $a \in Um (P_{3} \otimes_{R} B)$ and $f_{a}:Spec(B) \rightarrow E (P_{3})\setminus 0$ the morphism of schemes corresponding to $a$. Then $\mathcal{V}\circ f_{a}$ corresponds to the generalized Vaserstein symbol of $a$ associated to the trivialization $\theta_{0}\otimes_{R} B$.
\end{Lem}

\begin{proof}
This follows directly from Lemma \ref{Naturality} and the discussion of the universal case above.
\end{proof}

As a matter of fact, there is a formal way to prove that we can assume that the composite $E(P)\setminus 0 \xrightarrow{\bar{pr}_{n}^{-1}} \bar{Q}^{R}_{2n-1} \xrightarrow{\mathcal{V}} \mathcal{R}\Omega_{s}^{1}\mathcal{GW}^{3}$ can be represented by an actual morphism of pointed spaces $E(P)\setminus 0 \rightarrow \mathcal{R}\Omega_{s}^{1}\mathcal{GW}^{3}$: Since $\mathcal{R}\Omega_{s}^{1}\mathcal{GW}^{3}$ is $\mathbb{A}^{1}_{R}$-fibrant, we already know that the composite $E (P)\setminus 0 \xrightarrow{\bar{pr}_{n}^{-1}} \bar{Q}^{R}_{2n-1} \xrightarrow{\mathcal{V}} \mathcal{R}\Omega_{s}^{1}\mathcal{GW}^{3}$ is given by an actual morphism of spaces. Moreover, since the composite $Spec (R) \rightarrow E (P)\setminus 0 \xrightarrow{\mathcal{V}\bar{pr}_{n}^{-1}} \mathcal{R}\Omega_{s}^{1}\mathcal{GW}^{3}$ computes the generalized Vaserstein symbol of the projection $\pi_{R}: P_{0} \oplus R \rightarrow R$, it is null-homotopic. As $\mathcal{R}\Omega_{s}^{1}\mathcal{GW}^{3}$ is $\mathbb{A}^{1}_{R}$-fibrant, there is a naive $\mathbb{A}^{1}_{R}$-homotopy from the basepoint $Spec (R) \rightarrow \mathcal{R}\Omega_{s}^{1}\mathcal{GW}^{3}$ of $\mathcal{R}\Omega_{s}^{1}\mathcal{GW}^{3}$ to the composite $Spec (R) \rightarrow E (P)\setminus 0 \xrightarrow{\mathcal{V}\bar{pr}_{n}^{-1}} \mathcal{R}\Omega_{s}^{1}\mathcal{GW}^{3}$. By adjuntion, this naive $\mathbb{A}^{1}_{R}$-homotopy is actually represented by a morphism of spaces $H: Spec (R) \rightarrow \underline{Hom} (\mathbb{A}^{1}_{R}, \mathcal{R}\Omega_{s}^{1}\mathcal{GW}^{3})$.\\
As $\underline{Hom}(Spec (R), \mathcal{R}\Omega_{s}^{1}\mathcal{GW}^{3}) = \mathcal{R}\Omega_{s}^{1}\mathcal{GW}^{3}$, we obtain a commutative diagram

\begin{center}
$\begin{xy}
  \xymatrix{
      Spec (R) \ar[r]^-H \ar[d]    &   \underline{Hom} (\mathbb{A}^{1}_{R}, \mathcal{R}\Omega_{s}^{1}\mathcal{GW}^{3}) \ar[d]_{ev_{1}}  \\
      E (P) \setminus 0 \ar[r]_{\mathcal{V}\bar{pr}_{n}^{-1}}             &   \mathcal{R}\Omega_{s}^{1}\mathcal{GW}^{3},   
  }
\end{xy}$
\end{center}

where the right-hand vertical morphism is induced by evaluation at $1$. By \cite[Lemma 2.2.9]{MV}, this morphism is a simplicial fibration and weak equivalence; since the morphism $Spec (R) \rightarrow E (P) \setminus 0$ is a cofibration, there automatically exists a morphism $F: E (P) \setminus 0 \rightarrow \underline{Hom} (\mathbb{A}^{1}_{R}, \mathcal{R}\Omega_{s}^{1}\mathcal{GW}^{3})$ making the two resulting triangles commute. If we let $ev_{0}: \underline{Hom} (\mathbb{A}^{1}_{R}, \mathcal{R}\Omega_{s}^{1}\mathcal{GW}^{3}) \rightarrow \mathcal{R}\Omega_{s}^{1}\mathcal{GW}^{3}$ be the morphism induced by evaluation at $0$, then the composite $ev_{0} F$ is a pointed morphism $E (P) \setminus 0 \rightarrow \mathcal{R}\Omega_{s}^{1}\mathcal{GW}^{3}$ which is naively $\mathbb{A}^{1}_{R}$-homotopic to $\mathcal{V} \bar{pr}_{n}^{-1}$. Hence we can assume that the composite $E(P)\setminus 0 \xrightarrow{\bar{pr}_{n}^{-1}} \bar{Q}^{R}_{2n-1} \xrightarrow{\mathcal{V}} \mathcal{R}\Omega_{s}^{1}\mathcal{GW}^{3}$ can be represented by an actual morphism of pointed spaces $E(P)\setminus 0 \rightarrow \mathcal{R}\Omega_{s}^{1}\mathcal{GW}^{3}$.

\section{Results}\label{3}

In this section, we finally prove the main results of this paper. Our reinterpretation of the generalized Vaserstein symbol in the previous section enables us to prove the following sum formula:

\begin{Thm}\label{T3.1}
Let $k$ be a perfect field with $char (k) \neq 2$ such that $-1 \in {k^{\times}}^{2}$, $R$ a smooth affine $k$-algebra and $n \in \mathbb{N}$. Furthermore, let $P_{0}$ be a projective $R$-module of rank $2$ with a trivialization $\theta_{0}: R \xrightarrow{\cong} \det(P_{0})$ of its determinant. If $n \equiv 0,1~mod~4$, then $V_{\theta_{0}} (a_{0},a_{R}^{n}) = n \cdot V_{\theta_{0}} (a_{0},a_{R})$ for all $(a_{0},a_{R}) \in Um (P_{0} \oplus R)$.
\end{Thm}

\begin{proof}
As we have seen, the generalized Vaserstein symbol can be defined by means of a pointed morphism $\mathcal{V}: E (P_{3})\setminus 0 \rightarrow \mathcal{R}\Omega_{s}^{1} \mathcal{GW}^{3}$ in $Spc_{R,\bullet}$. We denote by $\mathcal{V}_{+}: {(E(P_{3})\setminus 0)}_{+} \rightarrow \mathcal{R}\Omega_{s}^{1} \mathcal{GW}^{3}$ the morphism obtained fom $\mathcal{V}$ by sending a formally added point to the basepoint of $\mathcal{R}\Omega_{s}^{1} \mathcal{GW}^{3}$.\\
The strategy of the proof is as follows: We want to compute $\mathcal{V}_{+} \circ \psi_{n}$, where $\psi_{n}:{(E(P_{3})\setminus 0)}_{+} \rightarrow {(E(P_{3})\setminus 0)}_{+}$ corresponds to the $n$-fold power operation. For this purpose, we extend the morphism $\mathcal{V}_{+}: {(E(P_{3})\setminus 0)}_{+} \rightarrow \mathcal{R}\Omega_{s}^{1} \mathcal{GW}^{3}$ to another space $\mathcal{Y}$ in order to obtain a commutative diagram of the form

\begin{center}
$\begin{xy}
  \xymatrix{
     {(E(P_{3})\setminus 0)}_{+} \ar[d]^{i} \ar[r]^{\psi_{n}} & {(E(P_{3})\setminus 0)}_{+} \ar[d]^{i} \ar[r]^{\mathcal{V}_{+}} & \mathcal{R}\Omega_{s}^{1}\mathcal{GW}^{3} \ar@2{-}[d]\\
     \mathcal{Y} \ar[r]^{\bar{\psi}_{n}} & \mathcal{Y} \ar[r]^{\bar{\mathcal{V}}} & \mathcal{R}\Omega_{s}^{1}\mathcal{GW}^{3}.}
\end{xy}$
\end{center}

By proving that $\bar{\mathcal{V}}\circ\bar{\psi_{n}}$ is equal to $n \cdot \bar{\mathcal{V}}$ in $\mathcal{H}_{\bullet} (R)$ with respect to the group structure on $[\mathcal{Y}, \mathcal{R}\Omega_{s}^{1}\mathcal{GW}^{3}]_{\mathbb{A}^{1}_{R},\bullet}$ induced by $\mathcal{R}\Omega_{s}^{1}\mathcal{GW}^{3}$ we then formally obtain that also $\mathcal{V}_{+} \circ \psi_{n}$ is equal to $n \cdot \mathcal{V}_{+}$ in $\mathcal{H}_{\bullet} (R)$ with respect to the group structure on $[{(E(P_{3})\setminus 0)}_{+}, \mathcal{R}\Omega_{s}^{1}\mathcal{GW}^{3}]_{\mathbb{A}^{1}_{R},\bullet}$ induced by $\mathcal{R}\Omega_{s}^{1}\mathcal{GW}^{3}$. The advantage of the space $\mathcal{Y}$ will be that it is a $\mathbb{P}^{1}_{R}$-suspension and an $h$-cospace, so by the usual Eckmann-Hilton argument we can equivalently consider the group structure on $[\mathcal{Y}, \mathcal{R}\Omega_{s}^{1}\mathcal{GW}^{3}]_{\mathbb{A}^{1}_{R},\bullet}$ induced by $\mathcal{Y}$ as an $h$-cospace and use Corollary \ref{P1Computation} in order to compute $\bar{\mathcal{V}}\circ\bar{\psi_{n}}$.\\
Now let us realize this strategy. Setting $P = P_{3}$, we consider the pushout square

\begin{center}
$\begin{xy}
  \xymatrix{
     {((E(P_{0})\setminus 0) \times \mathbb{G}_{m,R})}_{+} \ar[d] \ar[r] & {(E (P_{0}) \times \mathbb{G}_{m,R})}_{+} \ar[d] \\
     {((E (P_{0}) \setminus 0) \times \mathbb{A}_{R}^{1})}_{+} \ar[r] &  {(E (P) \setminus 0)}_{+}}
\end{xy}$
\end{center}

in $Spc_{R, \bullet}$ given by the Zariski covering of ${(E (P)\setminus 0)}_{+}$, which is also a homotopy pushout square. Furthermore, we also consider the square

\begin{center}
$\begin{xy}
  \xymatrix{
     {(E(P_{0})\setminus 0)}_{+} \times \mathbb{G}_{m,R} \ar[d] \ar[r] & E (P_{0}) \times \mathbb{G}_{m,R}\\
     {(E (P_{0}) \setminus 0)}_{+} \times \mathbb{A}_{R}^{1} }
\end{xy}$
\end{center}

and let $\mathcal{Y}$ be its homotopy pushout. Clearly, the obvious morphism from the first to the second diagram induces a morphism $i: {(E(P)\setminus 0)}_{+} \rightarrow \mathcal{Y}$. Furthermore, the $n$-fold power maps $\mathbb{A}^{1}_{R} \rightarrow \mathbb{A}^{1}_{R}$ and $\mathbb{G}_{m,R} \rightarrow \mathbb{G}_{m,R}$ induce a morphism from the first diagram to itself and hence define a morphism $\psi_{n}:  {(E(P)\setminus 0)}_{+} \rightarrow {(E(P)\setminus 0)}_{+}$ corresponding to the $n$-fold power operation. Analogously, the $n$-fold power maps induce a morphism from the second diagram to itself, which induces a morphism $\bar{\psi}_{n}: \mathcal{Y} \rightarrow \mathcal{Y}$. By sending $\ast \times \mathbb{G}_{m,R}$ and $\ast \times \mathbb{A}^{1}_{R}$ to the basepoint of $\mathcal{R}\Omega_{s}^{1}\mathcal{GW}^{3}$, we can extend the morphism $\mathcal{V}_{+}$ obtained from the morphism $\mathcal{V}$ defining the generalized Vaserstein symbol to a morphism $\bar{\mathcal{V}}: \mathcal{Y} \rightarrow \mathcal{R}\Omega_{s}^{1}\mathcal{GW}^{3}$.\\
Now the desired commutative diagram

\begin{center}
$\begin{xy}
  \xymatrix{
     {(E(P)\setminus 0)}_{+} \ar[d]^{i} \ar[r]^{\psi_{n}} & {(E(P)\setminus 0)}_{+} \ar[d]^{i} \ar[r]^{\mathcal{V}_{+}} & \mathcal{R}\Omega_{s}^{1}\mathcal{GW}^{3} \ar@2{-}[d]\\
     \mathcal{Y} \ar[r]^{\bar{\psi}_{n}} & \mathcal{Y} \ar[r]^{\bar{\mathcal{V}}} & \mathcal{R}\Omega_{s}^{1}\mathcal{GW}^{3}}
\end{xy}$
\end{center}

shows that it suffices to show that the composition $\bar{\mathcal{V}}\circ\bar{\psi_{n}}$ is equal to $n \cdot \bar{\mathcal{V}}$ in $\mathcal{H}_{\bullet} (R)$ with respect to the group structure on $[\mathcal{Y}, \mathcal{R}\Omega_{s}^{1}\mathcal{GW}^{3}]_{\mathbb{A}^{1}_{R},\bullet}$ induced by $\mathcal{R}\Omega_{s}^{1}\mathcal{GW}^{3}$.\\
But since $\mathcal{Y}$ is the homotopy pushout of the diagram

\begin{center}
$\begin{xy}
  \xymatrix{
     {(E(P_{0})\setminus 0)}_{+} \times \mathbb{G}_{m,R} \ar[d] \ar[r] & \mathbb{G}_{m,R}\\
     {(E (P_{0}) \setminus 0)}_{+}}
\end{xy}$,
\end{center}

the space $\mathcal{Y}$ is just given by the join $\mathbb{G}_{m,R}\ast{(E (P_{0}) \setminus 0)}_{+}$ (cf. \cite[p.219]{Mo}). The quotient morphism $\mathbb{G}_{m,R}\ast{(E (P_{0}) \setminus 0)}_{+}  \rightarrow \Sigma_{s}\mathbb{G}_{m,R}\wedge{(E(P_{0})\setminus 0)}_{+}$ is a weak equivalence and hence $\mathcal{Y}$ is weakly equivalent to $S^{1}\wedge\mathbb{G}_{m,R}\wedge{(E(P_{0})\setminus 0)}_{+}$ (cf. \cite[p.220]{Mo}) and therefore has the structure of an $h$-cogroup. Since the morphism $\bar{\psi_{n}}$ is the identity on ${(E (P_{0}) \setminus 0)}_{+}$ and the $n$-fold power map on $\mathbb{G}_{m,R}$, applying the quotient morphism ${(E (P_{0}) \setminus 0)}_{+} \ast \mathbb{G}_{m,R} \rightarrow \Sigma_{s}\mathbb{G}_{m,R}\wedge{(E(P_{0})\setminus 0)}_{+}$ yields that the morphism $\bar{\psi}_{n}$ then corresponds to the smash product of the $n$-fold power map on $\mathbb{G}_{m,R}$ with the identity on $S^{1}\wedge{(E(P_{0})\setminus 0)}_{+}$. By Corollary \ref{P1Computation}, this implies that $\bar{\psi}_{n}$ is equal to $n \cdot id_{\mathcal{Y}}$ in $[\mathcal{Y}, \mathcal{Y}]_{\mathbb{A}^{1}_{R},\bullet}$ and also that $\bar{\mathcal{V}}\circ\bar{\psi}_{n}$ is equal to $n \cdot \bar{\mathcal{V}}$ in $[\mathcal{Y}, \mathcal{R}\Omega_{s}^{1}\mathcal{GW}^{3}]_{\mathbb{A}^{1}_{R},\bullet}$ with respect to the group structures induced by $\mathcal{Y}$ as an $h$-cogroup. By the usual Eckmann-Hilton argument, it follows that $\bar{\mathcal{V}}\circ\bar{\psi}_{n}$ is equal to the $n$-fold sum of $\bar{\mathcal{V}}$ with respect to the group structure on $[\mathcal{Y}, \mathcal{R}\Omega_{s}^{1}\mathcal{GW}^{3}]_{\mathbb{A}^{1}_{R},\bullet}$ induced by $\mathcal{R}\Omega_{s}^{1}\mathcal{GW}^{3}$ as an $h$-group. This proves the theorem.
\end{proof}

As an application of the previous theorem, we can generalize a result of Fasel-Rao-Swan on transformations of unimodular rows via elementary matrices:

\begin{Thm}\label{T3.2}
Let $R$ be a normal affine $k$-algebra of dimension $d\geq 3$ over an algebraically closed field $k$ with $char(k) \neq 2$; if $d=3$, furthermore assume that $R$ is smooth. Let $P_{0}$ be a projective $R$-module of rank $2$ with a trivial determinant and let $P_{n} = P_{0} \oplus R^{n-2}$ for $n \geq 3$. Then for any $a \in Um (P_{d})$ and $j \in \mathbb{N}$ with $gcd (char(k),j) = 1$ there is an automorphism $\varphi \in E(P_{d})$ such that $a \varphi$ has the form $b = (b_{0},b_{3}^{j},...,b_{d})$.
\end{Thm}

\begin{proof}
Let $a = (a_{0},a_{3},...,a_{d}) \in Um (P_{d})$ and $I = \langle a_{4},...,a_{d} \rangle$. By Lemma \ref{SwanBertiniArgument}, we know that we can assume that $R/I$ is either $0$ or a smooth affine algebra of dimension $3$ over $k$. If $R/I = 0$, then Lemma \ref{ElementaryCompletion} proves the statement of the theorem. So let us assume that $R/I$ is a smooth affine algebra of dimension $3$ over $k$. In this case, we know that the generalized Vaserstein symbol associated to $P_{0}/IP_{0}$ and any fixed trivialization of its determinant gives a pointed bijection between $Um (P_{3}/IP_{3})/E(P_{3}/IP_{3})$ and $\tilde{V} (R/I)$; this bijection induces a group structure on $Um (P_{3}/IP_{3})/E(P_{3}/IP_{3})$ (cf. \cite[Theorem 4.1]{Sy1}). Since $\tilde{V} (R/I)$ is $l$-divisible for any prime $l$ with $gcd(l,char (k))=1$ and $char(k) \neq 2$ (cf. \cite[Propositions 5.1 and 6.1]{FRS}), there is $(\bar{b}_{0},\bar{b}_{3}) \in Um (P_{3}/IP_{3})$ with $4j \cdot (\bar{b}_{0},\bar{b}_{3}) = (\bar{a}_{0},\bar{a}_{3})$ in $Um (P_{3}/IP_{3})/E(P_{3}/IP_{3})$. Then the previous theorem implies that in fact $(\bar{b}_{0},\bar{b}_{3}^{4j}) = (\bar{a}_{0},\bar{a}_{3})$ in $Um (P_{3}/IP_{3})/E(P_{3}/IP_{3})$. Applying the map $\Phi_{3} (a)$ now yields the theorem.
\end{proof}

We can deduce the following corollaries from the previous theorem:

\begin{Kor}\label{C3.3}
Let $R$ be a smooth affine algebra of dimension $3$ over an algebraically closed field $k$ with $char (k) \neq 2$ and let $P_{0}$ a projective $R$-module of rank $2$ with trivial determinant. Then $Um (P_{0} \oplus R)/SL (P_{0} \oplus R)$ is trivial; in particular, $P_{0}$ is cancellative.
\end{Kor}

\begin{proof}
Let $a = (a_{0},a_{3}) \in Um (P_{0} \oplus R)$. By the previous theorem, there is $\varphi \in E (P_{0} \oplus R)$ such that $a \varphi$ is of the form $b = (b_{0}, b_{3}^{2})$. By \cite[Proposition 4.18]{Sy1}, there is $\psi \in SL (P_{0} \oplus R)$ such that $b \psi$ is the projection onto $R$. This proves that $Um (P_{0} \oplus R)/SL (P_{0} \oplus R)$ is trivial. In particular, this implies that the orbit space $Um (P_{0} \oplus R)/Aut (P_{0} \oplus R)$ is trivial and hence that $P_{0}$ is cancellative.
\end{proof}

\begin{Kor}\label{C3.4}
Let $R$ be a smooth affine algebra of dimension $4$ over an algebraically closed field $k$ with $char (k) \neq 2,3$ and let $P$ be a projective $R$-module of rank $3$ such that $c_{1} (P) = 0$, $c_{2} (P) = 0$ and $c_{3} (P) = 0$. Then $P$ is cancellative.
\end{Kor}

\begin{proof}
First of all, $P$ admits a free direct summand of rank $1$ because of the confirmation of Murthy's splitting conjecture by Asok-Fasel in \cite{AF2}. So write $P = P_{0} \oplus R$, where $P_{0}$ is a projective $R$-module of rank $2$ with a trivial determinant. As usual, we let $P_{n} = P_{0} \oplus Re_{3} \oplus ... \oplus Re_{n}$ for any $n \geq 3$ and we use the notation of Section \ref{2.1}. By Suslin's cancellation theorem (cf. \cite{S1}), any projective $R$-module stably isomorphic to $P_{0} \oplus R$ is the kernel of an epimorphism $a=(a_{0},a_{3},a_{4}) \in Um (P_{4})$. By Lemma \ref{SwanBertiniArgument}, we can assume that $R/a_{4}R$ is a smooth affine algebra of dimension $3$ over $k$. By Theorem \ref{T3.2}, there exists an automorphism $\varphi_{1} \in E(P_{4})$ such that $(b_{0},b_{3}^{3},b_{4}) = (a_{0}, a_{3}, a_{4}) \varphi_{1} \in Um (P_{4})$ and the proof of Theorem \ref{T3.2} shows that we can assume that $b_{4} = a_{4}$. By Lemma \ref{PowerOperations2}, there is $\varphi_{2} \in Aut (P_{4})$ such that $(b_{0},b_{3},b_{4}^{3}) = (b_{0},b_{3}^{3},b_{4}) \varphi_{2}$. Since $c_{1} (P_{0}) = 0$ and $c_{2} (P_{0}) = 0$ and $R/a_{4}R$ is a smooth affine threefold, it follows from \cite{AF1} that $P_{0}/a_{4}P_{0}$ is a free $R/a_{4}R$-module of rank $2$. Furthermore, any $R/a_{4}R$-linear epimorphism $P_{3}/a_{4}P_{3} \rightarrow R/a_{4}R$ can be completed to an $R/a_{4}R$-linear automorphism of $P_{3}/a_{4}P_{3}$ (this follows from \cite[Corollary 6.8]{AF1}). Therefore \cite[Lemma 2]{S1} implies that there is $\varphi_{3} \in Aut (P_{4})$ such that $\pi_{4,4} = (b_{0},b_{3},b_{4}^{3}) \varphi_{3}$. In particular, $\pi_{4,4} = (a_{0}, a_{3}, a_{4}) \varphi_{1} \varphi_{2} \varphi_{3}$, which finishes the proof.
\end{proof}

\end{document}